\documentclass[12pt]{siamart1116}

\usepackage{dcolumn}
\usepackage{bm}
\usepackage{subfigure}
\usepackage{comment}
\usepackage{enumerate}
\usepackage[amsmath,thmmarks]{ntheorem}
\usepackage{bbm}
\theorembodyfont{\normalfont}
\newcommand{\supp}{\operatorname{supp}}

\newcommand{\ihalf}[0]{{i+1/2}}
\newcommand{\imhalf}[0]{{i-1/2}}

\newcommand{\R}[0]{\mathbb{R}}
\newcommand{\N}[0]{\mathbb{N}}
\newcommand{\qqand}[0]{ \qquad{\text{and}} \qquad}
\newcommand{\qand}[0]{ \quad{\text{and}} \quad}

\newtheorem{remark}{Remark}

\renewcommand{\d}[0]{{\rm{d}}}

\usepackage{lipsum}
\usepackage{amsfonts}
\usepackage{graphicx}
\usepackage{epstopdf}
\usepackage{algorithmic}
\ifpdf
  \DeclareGraphicsExtensions{.eps,.pdf,.png,.jpg}
\else
  \DeclareGraphicsExtensions{.eps}
\fi

\numberwithin{theorem}{section}

\newcommand{\TheTitle}{Zoology of a non-local cross-diffusion model for two species} 
\newcommand{\TheAuthors}{J. A. Carrillo, Y. Huang, and M. Schmidtchen}

\headers{\TheTitle}{\TheAuthors}

\title{{\TheTitle}}

\author{
  Jos\'e A. Carrillo\thanks{Department of Mathematics, Imperial College London, SW7 2AZ London, UK.
    (\email{carrillo@imperial.ac.uk})}
  \and
  Yanghong Huang\thanks{School of Mathematics, The University of Manchester, Manchester M13 9PL, UK. (\email{yanghong.huang@manchester.ac.uk}).}
  \and
  Markus Schmidtchen\thanks{Department of Mathematics, Imperial College London, SW7 2AZ London, UK.}(\email{m.schmidtchen15@imperial.ac.uk})
}

\usepackage{amsopn}

\ifpdf
\hypersetup{
  pdftitle={\TheTitle},
  pdfauthor={\TheAuthors}
}
\fi

\begin{document}

\maketitle

\begin{abstract}
We study a non-local two species cross-interaction model with cross-diffusion. We propose a positivity preserving finite volume scheme based on the numerical method introduced in Ref. \cite{CCH15} and explore this new model numerically in terms of its long-time behaviours. Using the so gained insights, we compute analytical stationary states and travelling pulse solutions for a particular model in the case of attractive-attractive/attractive-repulsive cross-interactions.  We show that, as the strength of the cross-diffusivity decreases, there is a transition from adjacent solutions to completely segregated densities, and we compute the threshold analytically for attractive-repulsive cross-interactions. Other bifurcating stationary states with various coexistence components of the support are analysed in the attractive-attractive case. We find a strong agreement between the numerically and the analytically computed  steady states in these particular cases, whose main qualitative features are also present for more general potentials. 
\end{abstract}

\begin{keywords}
cross-diffusion, non-local aggregation-diffusion systems, volume exclusion
\end{keywords}

\begin{AMS}
   	35K55, 65N08, 35C05
\end{AMS}

\section{\label{sec:intro}Introduction}
Multi-agent systems in nature oftentimes exhibit emergent behaviour, \textit{i.e.} the formation of patterns in the absence of a leader or external stimuli such as light or food sources. The most prominent examples of these phenomena are probably fish schools, flocking birds, and herding sheep, reaching across scales
from tiny bacteria to huge mammals.

While self-interaction models for one particular species have been extensively studied, cf. Refs. \cite{OCBC06, CS07, CFTV10, MEK99, TBL06} and references therein, there has been a growing interest in understanding and modelling interspecific interactions, \textit{i.e.} the  interaction among different types of species. One way to derive macroscopic models from microscopic dynamics consists in taking suitable scaling limits as the number of individuals goes to infinity. Minimal models for collective behaviour include attraction and/or repulsion between individuals as the main source of interaction, see \cite{TBL06,CDP,CFTV10,review} and the references therein. Attraction and repulsion are normally introduced through effective pairwise potentials whose strength and scaling properties determine the limiting continuum equations, see \cite{Ol,CBM,BV,meanfield}. 
Usually localised strong repulsion gives rise to non-linear diffusion like those in porous medium type models~\cite{Ol}, while long-range attraction remains non-local in the final macroscopic equation,
 see \cite{meanfield} and the references therein.

In this paper we propose a finite-volume scheme to study two-species systems of the form
\begin{subequations}
\label{eq:OurModelMultiD}
\begin{align}
    \partial_t \rho &= \nabla \cdot\Big(\rho \nabla \big(W_{11}\star\rho + W_{12}\star\eta + \epsilon (\rho+\eta)\big)\Big),\\
    \partial_t \eta &= \nabla \cdot\Big(\eta \nabla \big(W_{22}\star\eta + W_{21}\star\rho + \epsilon (\rho+\eta)\big)\Big),
\end{align}
with given initial data
\begin{align}
	\rho(x,0) = \rho_0(x), \qquad \text{and} \qquad \eta(x,0) = \eta_0(x).
\end{align}
\end{subequations}
Here, $\rho, \eta$ are two unknown mass densities, $W_{11}, W_{22}$ are self-interaction potentials  (or intraspecific interaction potentials), $W_{12},W_{21}$ are cross-interaction potentials (or interspecific interaction), and $\epsilon>0$ is the coefficient of the cross-diffusivity. The non-linear diffusion term of porous medium type can be considered as a mechanism to include volume exclusion in cell chemotaxis \cite{Hillen01,Hillen02,CC}, since it corresponds to very concentrated repulsion between all individuals.

This model can also be easily understood as a natural extension of the well-known aggregation equation (cf. \cite{MEK99,TBL06,BCL,CDFLS11} ) to two species including a cross-diffusion term. Common interaction potentials for the one species case include power laws $W(x)=|x|^p/p$, as for instance in the case of granular media models, cf. \cite{BCP97, Tos00}. Another choice is a combination of power laws of the form $W(x)=|x|^a/a- |x|^b/b$, for $-N<b<a$ where $N$ is the space dimension. These potentials, featuring short-range repulsion and long-range attraction, are typically chosen in the context of swarming models, cf. \cite{KSUB,BCLR2,FH,FHK,BKSUV,CHM14,CDM16}. Other typical choices include characteristic functions of sets (like spheres) or Morse potentials 
$$
	W(x) = -c_a \exp(-|x|/l_a) + c_r \exp(-|x|/l_r),
$$
or their regularised versions 
$
W_p(x) = -c_a \exp(-|x|^p/l_a) + c_r \exp(-|x|^p/l_r),
$
where $c_a,c_r$ and $l_a,l_r$ denote the interaction strength and radius of the attractive (resp. repulsive) part and $p\geq2$,  cf. \cite{OCBC06, CMP13, CHM14}. These potentials display a decaying interaction strength, \textit{e.g.} accounting for biological limitations of visual, acoustic or olfactory sense. The asymptotic behaviour of solutions to one single equation where the repulsion is modelled by non-linear diffusion and the attraction by non-local forces has also received lots of attention in terms of qualitative properties, stationary states and metastability, see \cite{BFH,CCH15,EK,CCH1,CCH2} and the references therein.

Systems without cross-diffusion, $\epsilon=0$, were proposed in \cite{DFF13} as the formal mean-field limit of the following ODE system 
\begin{align*}
	\dot x_i &= - \frac1N \sum_{j\neq i}W_{11}(x_i-x_j) - \frac1M\sum_{j\neq i} W_{12}(x_i-y_j),\\
	\dot y_i &= - \frac1M \sum_{j\neq i}W_{22}(y_i-y_j) - \frac1N\sum_{j\neq i} W_{21}(y_i-x_j).
\end{align*}
For symmetrisable systems, \textit{i.e.} systems such that there exists some positive constant $\alpha>0$ with $W_{12}=\alpha W_{21}$, they show the system can be assigned an interaction energy functional. As a result, the system admits a gradient flow structure and variational schemes can be applied to ensure existence of solutions, cf. \cite{DFF13, JKO98}. However, in many contexts such a condition is too exclusive in the sense that lots of applications exhibit a lack of symmetry in the interactions between different species. 

In order to treat the system for general, and possibly different, cross-interactions $W_{12},W_{21}$, they modify the well-known variational scheme to prove convergence even in the absence of gradient flow structure. 
These systems without cross-diffusion appear in modelling cell adhesion in mathematical biology with applications in zebrafish patterning and tumour growth models, see \cite{GC,DTGC,PBSG,VS} for instance. 

In this paper we extend their cross-interaction model by a cross-diffusion term which is used to take into account the population pressure, \textit{i.e.} the tendency of individuals to avoid areas of high population density. As cross-diffusion we choose the form introduced by Gurtin and Pipkin in their seminal paper \cite{GP84}. Although their work is antedated by results of mathematicians and biologists interested in density segregation effects of biological evolution equations, cf. \cite{SKT,Shi80} and references therein, the particularity about their population pressure model is the occurrence of strict segregation of densities under certain circumstances, cf. \cite{GP84, BGH87, BGH87a}. This cross-diffusion term has been the basis to incorporate volume exclusions in models for e.g. tumour growth \cite{BDPM10} or cell adhesion \cite{MT15}. 

Hence, our model is of particular interest from a modelling point of view taking into account non-local interactions between the same species and different species as well as the urge of both species to avoid clustering. We discover a rich asymptotic behaviour including phenomena such as segregation of densities, regions of coexistence, travelling pulses -- all of which are observed in biological contexts, cf. \cite{SCBBSP10, TKWKMTT11}. Existence of segregated stationary states under certain assumptions on the interaction potentials for small cross-diffusivity has been very recently obtained in \cite{BDFS}. Here we show that it is in fact possible to find explicit stationary states and travelling pulses for certain singular not necessarily decaying interaction potentials showing coexistence and segregation of densities.

The rest of this paper is organised as follows: in Section \ref{sec:system} we discuss the basic properties of the system \eqref{eq:OurModelMultiD} in one dimension, in Section \ref{sec:schemes}  we propose our numerical scheme which is used in Section \ref{sec:numerical_results} to explore the model and its long-time behaviour numerically. These insights are used to make reasonable assumptions on the support of the asymptotic solutions in order to derive analytic expressions for their shape and give a first classification of the zoology of the different stationary states. Finally we discuss in Section
\ref{sec:gen} how generic these phenomena are for different potentials and we draw the final conclusions of this work in Section \ref{sec:con}.

\section{\label{sec:system}A non-local cross-diffusion model for two species}
Throughout this paper we consider system \eqref{eq:OurModelMultiD} in one spatial dimension. Then the  model reads
\begin{subequations}
\label{eq:our_system}
\begin{align}
    \partial_t \rho &= \partial_x (\rho \partial_x (W_{11}\star\rho + W_{12}\star\eta + \epsilon (\rho+\eta))),\label{eq:a_our_system}\\
    \partial_t \eta &= \partial_x (\eta \partial_x (W_{22}\star\eta + W_{21}\star\rho + \epsilon (\rho+\eta))),\label{eq:b_our_system}
\end{align}
\end{subequations}
for some initial data $\rho(x,0) = \rho_0(x)$, and $\eta(x,0) = \eta_0(x)$,
and radially symmetric potentials  $W_{ij}$, for $i,j= 1,2$. We can obtain some apriori estimates on solutions by using the following energy
\begin{align*}
	\mathcal{E}(\rho, \eta) = \int_{\R } \rho \log \rho\; \d x + \int_{\R } \eta \log \eta\;\d x.
\end{align*}
We note that for $W_{ij} \in W^{2,\infty}(\R)$,  along any solution $(\rho,\eta)$ of system \eqref{eq:our_system}, there holds
\begin{align*}
	\frac{\d }{\d t} \mathcal{E}(\rho, \eta) 
	=&\,-\int \epsilon [\partial_x (\rho + \eta)]^2 \d x\\
	 &-\int \partial_x \rho\, \partial_x\big(W_{11}\star\rho + W_{12}\star \eta\big)\d x
	 -\int \partial_x \eta\, \partial_x\big(W_{22}\star\eta + W_{21}\star \rho\big)\d x\\
     \leq &\, - \int \epsilon |\partial_x (\rho + \eta)|^2 \d x \\
     &+ \int \rho \partial_x^2\big(W_{11} \star \rho + W_{12}\star \eta\big) \d x + \int \eta \partial_x^2\big(W_{22}\star \eta + W_{21}\star \rho\big)\d x,
\end{align*}
that is
\begin{align*}
	\frac{\d }{\d t} \mathcal{E}(\rho, \eta) \leq C - \int_{\R } \epsilon |\partial_x(\rho + \eta)|^2 \d x.
\end{align*}
In the case of $W_{ii}= C_{ii} x^2/2$ and $W_{ij}=\pm C_{ij}|x|$ for $i\neq j$ with non-negative constants $C_{ij}$, the estimate is also true, since 
\begin{align*}
	\int \rho \partial_x^2 W_{11}\star \rho
 \d x = C_{11} m_1\!\! \int \rho  \d x = C_{11} m_1^2, \quad
 	\left|\int \rho \partial_x^2 W_{ij} \star \eta \d x\right|  \leq 2C_{ij} \int \rho \eta < \infty,
 \end{align*}
and similarly for the terms in \eqref{eq:b_our_system}, as long as $\rho, \eta \in L^\infty(0,T;L^\infty(\R))$. Thus the terms
\begin{align*}
\int \rho \partial_x^2(W_{11} \star \rho + W_{12}\star \eta) \d x, 
\qquad
\mbox{and similarly}
\qquad
\int \eta \partial_x^2(W_{22}\star \eta + W_{21}\star \rho)\d x
\end{align*}
are bounded. We conclude that 
\begin{align*}
	\frac{\d }{\d t} \mathcal{E}(\rho, \eta) \leq C - \int_{\R } \epsilon |\partial_x(\rho + \eta)|^2 \d x,
\end{align*}
implying that $\rho+ \eta \in L^2(0,T;H^1(\R))$. We deduce that the sum of both species remains continuous for almost all positive times --- a property we will make use of later.  Now, let us introduce our notion of steady states.

\begin{definition}[\label{def:steadystate}Steady states]
	A pair of functions $(\rho,\eta)$ defined on $\mathbb{R}$ is called a steady state to \eqref{eq:our_system}, if both functions are integrable and bounded, $\rho, \eta \in L^1(\R)\cap L^\infty(\R)$, such that their sum satisfies  $\sigma:=\rho+ \eta \in H^1(\R)$ and there holds
\begin{align*}
	0 &= \partial_x (\rho \partial_x (W_{11}\star\rho + W_{12}\star\eta + \epsilon (\rho+\eta))),\\
	0 &= \partial_x (\eta \partial_x (W_{22}\star\eta + W_{21}\star\rho + \epsilon (\rho+\eta))),
\end{align*}
in the distributional sense.
\end{definition}

\begin{proposition}[Almost characterisation of steady states]
Any pair of functions $\rho, \eta\in L^1(\R)\cap L^\infty(\R)$ satisfying $\rho+\eta \in H^1(\R)$ such that any connected component of their supports has non-empty interior is a steady state of system \eqref{eq:our_system} if and only if there exist constants $c_1,c_2\in\R$, possibly different on different connected components of the supports, such that 
	\begin{align}
		\label{eq:steadystatecond}
		\begin{split}
		c_1 &= W_{11}\star \rho + W_{12}\star\eta + \epsilon(\rho + \eta),\\
		c_2 &= W_{22}\star \eta + W_{21}\star\rho + \epsilon(\rho + \eta).
		\end{split}
	\end{align}
\end{proposition}

\begin{proof}
	Clearly, the characterisation is sufficient, since the velocity field vanishes in each connected component of their supports if there exist constants $c_1,c_2$ such that Eqs.  \eqref{eq:steadystatecond} are satisfied.
	Conversely, if there holds
	\begin{align*}
	0 &= \partial_x (\rho \partial_x (W_{11}\star\rho + W_{12}\star\eta + \epsilon (\rho+\eta))),\\
	0 &= \partial_x (\eta \partial_x (W_{22}\star\eta + W_{21}\star\rho + \epsilon (\rho+\eta))),
	\end{align*}
 	we note that $\rho, \eta, \partial_x (\rho + \eta)\in L^2(\R)$ by the definition of steady state, and therefore the right-hand sides are distributional derivatives of $L^1$ functions. By a well-known result (cf. \textit{e.g.} \cite{JLJ98}, Lemma 1.2.1.), we deduce that there exist constants $K_1,K_2\in \R$ such that 
 	\begin{align*}
		K_1 &= \rho \partial_x (W_{11}\star\rho + W_{12}\star\eta + \epsilon (\rho+\eta)),\\
		K_2 &= \eta \partial_x (W_{22}\star\eta + W_{21}\star\rho + \epsilon (\rho+\eta)).
	\end{align*}
 Due to the integrabilty properties of the right-hand sides above, we infer that $K_1=K_2=0$, and thus in the interior of any connected component of the supports of $\rho$ and $\eta$, we obtain that there exist constants $c_1,c_2\in \R$ such that 
 	\begin{align*}
		c_1 &= W_{11}\star \rho + W_{12}\star\eta + \epsilon(\rho + \eta),\\
		c_2 &= W_{22}\star \eta + W_{21}\star\rho + \epsilon(\rho + \eta),
	\end{align*}
using the same argument as above.
\end{proof}

Note that the assumption on the interiors of the supports of the species is purely technical and due to the regularity assumptions on our definition of stationary states. This avoids pathological cases such as functions supported on a fat Cantor set.

\section{\label{sec:schemes}Numerical scheme}
In order to solve system \eqref{eq:our_system}, we introduce 
a finite vo\-lu\-me scheme based on Ref. \cite{CCH15}. 
The problem is posed on the domain $\Omega:=[-L,L]$ which is divided into $N$ equal control volumes $(C_i)_{i=1,\ldots, N}$, with $C_i:=[x_{i-1/2},x_{i+1/2}]$ and uniform size $\Delta x:=x_{i+1/2}-x_{i-1/2}$.  Finally, the time interval $[0,T]$ is discretised by $t^n = n \Delta t$, for $n=0,\ldots,\lceil T / \Delta t \rceil $. We define the discretised initial data via
\begin{align*}
	\rho_i^0 := \frac{1}{\Delta x}\int_{C_i} \rho_0(x)\d x, \qquad \text{and}\qquad 	\eta_i^0 := \frac{1}{\Delta x}\int_{C_i} \eta_0(x)\d x.
\end{align*}

We integrate system \eqref{eq:our_system} over the test cell $[t^n,t^{n+1}]\times C_i$ to obtain
\begin{align*}
	\frac{1}{\Delta x} \int_{C_i} \rho(t^{n+1}, x) \d x  &= \frac{1}{\Delta x} \int_{C_i}  \rho(t^{n}, x)\d x  - \frac{1}{\Delta x} \big(\bar F_{i+1/2}^{n} - \bar F_{i+1/2}^{n}\big),\\
	\frac{1}{\Delta x} \int_{C_i} \eta(t^{n+1}, x) \d x  &= \frac{1}{\Delta x} \int_{C_i}  \eta(t^{n}, x)\d x  - \frac{1}{\Delta x} \big(\bar G_{i+1/2}^{n} - \bar G_{i+1/2}^{n}\big),
\end{align*}
where $\bar F_\ihalf^n, \bar G_\ihalf^n$ denote the flux on the boundary of cell $C_i$, \textit{i.e.}
\begin{align}
	\label{eq:ctsfluxes}
	\begin{split}
	\bar F_\ihalf^n &:= -\int_{t^n}^{t^{n+1}} \big(\rho \partial_x (W_{11}\star\rho + W_{12}\star\eta + \epsilon(\rho + \eta) )\big)(x_\ihalf,t) \d t,\\
	\bar G_\ihalf^n &:= -\int_{t^n}^{t^{n+1}} \big(\eta \partial_x (W_{22}\star\eta + W_{21}\star\rho + \epsilon(\rho + \eta) )\big)(x_\ihalf,t) \d t.
	\end{split}
\end{align}
Then the finite volume scheme  for the cell averages $\rho_i^n$ and $\eta_i^n$ reads
\begin{subequations}
\label{eq:scheme}
\begin{align}
	\label{eq:timestep}
	\begin{split}
	\rho_i^{n+1} = \rho_i^n - \frac{\Delta t}{\Delta x} \left(F_\ihalf^n - F_\imhalf^n\right), \\
	\eta_i^{n+1} = \eta_i^n - \frac{\Delta t}{\Delta x} \left(G_\ihalf^n - G_\imhalf^n\right),
	\end{split}
\end{align}
where we approximate the fluxes on the boundary, Eqs.\eqref{eq:ctsfluxes}, by the numerical fluxes
\begin{align}\label{eq:numfluxes}
\begin{split}
	F_\ihalf^{n} = (U_\ihalf^{n})^+ \rho_i^{n} + (U_\ihalf^{n})^- \rho_{i+1}^{n},\\
	G_\ihalf^{n} = (V_\ihalf^{n})^+ \eta_i^{n} + (V_\ihalf^{n})^- \eta_{i+1}^{n},
\end{split}
\end{align}
using $(\cdot)^+:=\max(\cdot, 0)$ and $(\cdot)^-:=\min(\cdot, 0)$ to denote the positive part and the negative part, respectively. The velocity is discretised by centred differences:
\begin{align}
\label{eq:numvelocities}
	U_\ihalf^{n} = -\dfrac{\xi_{i+1}^{n} - \xi_i^{n}}{\Delta x}, \qquad \text{and} \qquad V_\ihalf^{n} = -\dfrac{\zeta_{i+1}^{n} - \zeta_i^{n}}{\Delta x}.
\end{align}
Here we have set
\begin{align}
	\label{eq:xi_direct}
	\begin{split}
	\xi_i^{n} &:=	\Delta x \sum_k \big(W_{11}^{i-k} \rho_k^{n} +  W_{12}^{i-k} \eta_k^{n}\big)\, + \epsilon (\rho_i^{n}+\eta_i^{n}),\\
	\zeta_i^{n} &:=	\Delta x \sum_k \big(W_{22}^{i-k} \eta_k^{n} +  W_{21}^{i-k} \rho_k^{n}\big)\, +\epsilon (\rho_i^{n}+\eta_i^{n}),
	\end{split}
\end{align}
where $W_{ij}^{l-k} = W_{ij}(x_l - x_k)$, for $i,j=1,2$.
\end{subequations}
This scheme has proven very robust for one species, and under a (more restrictive) CFL condition we can also prove the following result.

\begin{proposition}[Non-negativity preservation]
Consider system (\ref{eq:our_system}) with initial da\-ta $\rho_{0}, \eta_0\geq0$. Then for all $n \in \N$ the cell averages obtained by the finite volume method \eqref{eq:scheme} satisfy
$
	\rho_i^{n}, \eta_i^n \geq 0,
$
granted that the following CFL condition is satisfied
\begin{align*}
	\Delta t \leq \frac{\Delta x}{2 \max (\|U^n\|_\infty,  \|V^n\|_\infty)}.
\end{align*}
\end{proposition}

\begin{proof}
Let us assume that $\rho_i^n,\eta_i^n\geq 0$, and we need to show that then $\rho_i^{n+1},\eta_i^{n+1}\geq 0$. 
According to Eqs. (\ref{eq:timestep}, \ref{eq:numfluxes}) we have
\begin{align*}
\rho_i^{n+1}  &= 	\rho_i^{n}  - \frac{\Delta t}{\Delta x} (F_\ihalf^n -F_\imhalf^{ n})\\
&= 	\rho_i^{n} -\frac{\Delta t}{\Delta x}  \bigg[
(U_\ihalf^{n})^+ \rho_i^{n} + (U_\ihalf^{n})^- \rho_{i+1}^{n}- (U_\imhalf^{n})^+ \rho_{i-1}^{n} - (U_\imhalf^{n})^- \rho_i^{n} \bigg].
\end{align*}

We can rearrange the terms so that
\begin{align*}
	\rho_i^{n+1}  &=  	\rho_i^{n} \left(1 - \frac{\Delta t}{\Delta x} \left[(U_\ihalf^{n})^+  - (U_\imhalf^{n})^-\right] \right) \\
	&\qquad + \frac{\Delta t}{\Delta x}(U_\imhalf^{n})^+	\rho_{i-1}^{n} - \frac{\Delta t}{\Delta x}  (U_\ihalf^{n})^- \rho_{i+1}^{n}.
\end{align*}
Clearly, all terms in the second line are non-negative. The first line is non-negative if the CFL condition is satisfied. Application of the same procedure to $\eta_i^{n+1}$ yields the statement.
\end{proof}

\section{\label{sec:numerical_results}Numerical study}
In this section we study system \eqref{eq:our_system} numerically with  emphasis on its long time behaviour. Throughout this chapter we use the self-interaction potentials
\begin{align*}
	W_{11}(x) := W_{22}(x) := x^2/2 \approx 1 - e^{-|x|^2/2} \mbox{ near zero},
\end{align*}
and the cross-interaction potentials
\begin{align*}
	W_{12}(x) =  W_{21}(x) = |x|\approx 1-e^{-|x|} \mbox{ near zero and } W_{12}(x)  = |x|  = -W_{21}(x),
\end{align*}
for the interspecific attractive-attractive and attractive-repulsive case, respectively. This choice of potentials allows us to compute steady states of system \eqref{eq:our_system} explicitly. We find a wide range of different behaviours and properties, including segregation phenomena, for different cross-diffusivities and cross-interactions. Notice that the system is translationally invariant and therefore, if for symmetry considerations we can show that the centres of mass of both species in a stationary state are fixed and equal to some particular value, we can suppose that value to be zero without loss of generality.

From numerical simulations we observe that steady states are compactly supported which motivates this ansatz when computing the profiles analytically. This is also due to the non-linear diffusion of porous medium type in the volume exclusion term. This chapter is subdivided into two sections addressing the mutually attractive case and the attractive-repulsive case, respectively.

\subsection{\label{sec:attrattr}Attractive-attractive case}
Let us begin with the case of attractive interaction between both species, \textit{i.e.} $W_{12} = W_{21}= |x|$. Upon exploring the system numerically, we find a vast variety of stationary patterns, including 
both symmetric and non-symmetric profiles whose occurrence and stability depends on the cross-diffusivity.  

In fact, the coefficient $\epsilon$ of the cross-diffusivity plays a crucial role in the bifurcations 
of these profiles, and will be discussed in the next section.
Then, we study the system as the cross-diffusivity tends to zero and the stability of the steady states --- a matter that  seems closely intertwined with the bifurcations.

\subsubsection{Steady states and behavioural bifurcation}
We begin by introducing the two types of symmetric steady states observed in the attractive-attractive case. Motivated by numerical simulations, we assume that the stationary distributions are compactly supported, \textit{i.e.},
$$
	\supp(\rho) = [-c,c], \qquad \text{and} \qquad \supp(\eta) = [-b,b].
$$
where $0 < b \leq c$. The domain $[-c,-b]\cup[b,c]$ is then only
inhabited by the first species, but not $\eta$. Upon rearranging Eq. \eqref{eq:steadystatecond}, we obtain
\begin{align}
\label{eq:200616_1637}
	\rho(x) = -\frac1\epsilon \big(W_{11} \star \rho + W_{12}\star \eta  - c_1\big).
\end{align}
The two non-local terms $W_{11} \star \rho$ and $W_{12}\star \eta$ can be computed individually. 
First the self-interaction terms becomes
\begin{align}
	\label{eq:SelfInteractionTerm}
	W_{11}\star \rho (x) &= \int W_{11}(x-y) \rho(y)\d y = \frac12m_1 x^2 - M_1 x + \frac12 \bar M_1,
\end{align}
where 
\begin{align*}
m_1 = \int_{\R} \rho(x)\d x, \quad M_1 = \int_{\R} x \rho(x) \d x, \quad \text{and} \quad \bar M_1 = \int_{\R} x^2 \rho(x)\d x,
\end{align*}
are the mass and the first two moments of $\rho$, respectively. Then the cross-interaction term becomes
\begin{align}
	\label{eq:CrossInteractionTerm}
	W_{12}\star \eta(x) = 	\int |x-y| \eta(y)\d y = 
	\left\{
	\begin{array}{ll}
		M_2 - m_2 x, & x\in [-c,-b],\\[0.5em]
		m_2 x - M_2, & x\in  [b,c],
	\end{array}
	\right.
\end{align}
where $m_2, M_2$ denote the mass and the centre of mass of the second species. 
Due to symmetry and translational invariance of the solution, both $M_1$ and $M_2$ 
can be taken as zero without loss of generality. Upon substitution of the non-local terms in~\eqref{eq:SelfInteractionTerm} 
and~\eqref{eq:CrossInteractionTerm}, Eq. \eqref{eq:200616_1637} is simplified into
\begin{align*}
	\rho(x) = -\frac1\epsilon \bigg(\frac12m_1 x^2 + \frac12 \bar M_1 \pm ( - m_2 x) - c_1\bigg),
\end{align*}
where `$+$' is for the case $x \in [-c,-b]$, and `$-$' for $x \in [b,c]$, respectively. 
Using $\rho(\pm c) = 0$ at the boundary (where $\rho+\eta$ vanishes identically), we get
\begin{align}
	\label{eq:210616_0622}
	\rho(x) = 
	\left\{
      \begin{array}{ll}
        -\dfrac{1}{2\epsilon} m_1 (x^2 -c^2) + \dfrac{m_2}{\epsilon}
        (x+c), & x\in [-c,-b],\\[1em]
        -\dfrac{1}{2\epsilon} m_1 (x^2 -c^2) - \dfrac{m_2}{\epsilon}
        (x-c), & x \in [b,c].
		\end{array}
	\right.
\end{align}

Finally, let us consider the interval $[-b,b]$ where both species coexist. Again, $\rho$  satisfies
\begin{align}
	\label{eq:251016_1035}
	\begin{split}
	c_1 &= W_{11} \star \rho + W_{12} \star \eta + \epsilon(\rho + \eta)\\
	&=\frac12 m_1 x^2 + \frac12 \bar M_1 + \epsilon(\rho + \eta) + \int_{-b}^b |x-y| \eta(y)\d y,
	\end{split}
\end{align}
where the cross-interaction term $W_{12}\star\eta$ can be further reduced, according to
\begin{align*}
	\int_{-b}^b |x-y| \eta(y)\d y =  x\int_{-b}^x \eta(y) \d y - \int_{-b}^x y\eta(y)\d y + \int_x^b y\eta(y)\d y -x\int_x^b \eta(y)\d y.
\end{align*}
Notice that all terms on the right side are twice differentiable. Therefore from~\eqref{eq:251016_1035}, $\rho + \eta$ is twice differentiable in $(-b,b)$, and upon differentiating Eq. \eqref{eq:251016_1035} twice we obtain
\begin{equation}
\label{eq:200616_1714}
	0 = (\rho + \eta)'' + \frac2\epsilon\eta + \frac{m_1}{\epsilon},
\end{equation}
and similarly from the second equation in~\eqref{eq:steadystatecond}
\begin{equation}
\label{eq:200616_1714_1}
	0 = (\rho + \eta)'' + \frac2\epsilon\rho + \frac{m_2}{\epsilon}.
\end{equation}
The system of equations~\eqref{eq:200616_1714} and~\eqref{eq:200616_1714_1} can be solved 
by first introducing the decoupled system for $u:= \rho + \eta$ and $v:=\rho-\eta$, giving by
\begin{align*}
	2u'' + \frac2\epsilon u + \frac{m_1 + m_2}{\epsilon} &= 0,\\
	\frac2\epsilon v + \frac{m_2 - m_1}{\epsilon} &= 0.
\end{align*}
Thus, the solutions $\rho$ and $\eta$ are obtained as
\begin{align}
	\label{eq:200616_1716}
	\rho(x) =\frac{\hat u_1}{2} \sin\left(\frac{x}{\sqrt{\epsilon}}\right) + \frac{\hat u_2}{2}\cos\left(\frac{x}{\sqrt{\epsilon}}\right) - \frac{m_2}{2},\\
		\label{eq:200616_1716_1}
	\eta(x) =\frac{\hat u_1}{2} \sin\left(\frac{x}{\sqrt{\epsilon}}\right) + \frac{\hat u_2}{2}\cos\left(\frac{x}{\sqrt{\epsilon}}\right) - \frac{m_1}{2}.
\end{align}
In fact, due to symmetry there holds $\hat u_1 = 0$, and Eqs.(\ref{eq:200616_1716},\ref{eq:200616_1716_1}) can be simplified to 
\begin{align}\label{eq:200616_1726}	
	\rho(x) = \frac{\hat u_2}{2}\cos\left(\frac{x}{\sqrt{\epsilon}}\right) - \frac{m_2}{2},\qquad
	\eta(x) = \frac{\hat u_2}{2}\cos\left(\frac{x}{\sqrt{\epsilon}}\right) - \frac{m_1}{2}.
\end{align}
Hence the symmetric steady states are determined uniquely by  three parameters, $\hat u_2$, $b$ and $c$,
which are governed by algebraic equations.
Since $\eta$ is only supported on $[-b,b]$, the condition for the total mass of $\eta$ becomes
\begin{equation*}
	m_2 = \int_{-b}^b \eta(x)\d x = \sqrt{\epsilon}\hat{u}_2\sin \left(\frac{b}{\sqrt{\epsilon}}\right)-m_1b,
\end{equation*}
which yields
\begin{equation*}
	\hat u_2 = \frac{m_2 + m_1 b}{\sqrt{\epsilon} \sin\left(\dfrac{b}{\sqrt{\epsilon}}\right)}.
\end{equation*}

From  Eqs.(\ref{eq:210616_0622}, \ref{eq:200616_1726}), the condition for the total mass of $\rho$ becomes
\begin{multline}
\label{eq:200616_1728}
	m_1 = \left(\int_{-c}^{-b} +\int_{-b}^{b} +\int_{b}^{c}\right) \rho(x)\d x \cr
	=\frac{m_1(b+2c)(c-b)^2}{3\epsilon}+\frac{m_2(c-b)^2}{\epsilon} 
	+\sqrt{\epsilon}\hat{u}_2\sin \left(\frac{b}{\sqrt{\epsilon}}\right)-m_2b.
\end{multline}
When $\hat u_2$ is eliminated, Eq. (\ref{eq:200616_1728}) provides a relation between $b$ and $c$, \textit{i.e.},
\begin{equation}
	\label{eq:200616_1730_1}
	3 \epsilon m_1 + 3 \epsilon m_2 b = 3 \epsilon (m_2 + m_1 b) + (c - b)^2 (3 m_2 + 2 c m_1 + m_1 b).
\end{equation}
Finally, consider the continuity of the sum of the densities $\rho+\eta$ 
at $x=b$ (cf. Definition \ref{def:steadystate}), 
\begin{equation*}
	\lim_{x\uparrow b} \big(\rho(x) + \eta(x)\big) = \lim_{x \downarrow b} \big(\rho(x) + \eta(x)\big),
\end{equation*}
whence
\begin{equation}
\label{eq:200616_1730_2}
m_1 (c^2 - b^2) + \epsilon m_1 + \epsilon m_2 + 2 m_2 (c - b) = 2 \sqrt{\epsilon}  (m_2 + m_1 b) \cot\left(\dfrac{b}{\sqrt{\epsilon}}\right).
\end{equation}

Therefore $b$ and $c$ are in the zero locus of Eqs. (\ref{eq:200616_1730_1}, \ref{eq:200616_1730_2}) that are numerically solved, cf. Figure \ref{fig:AttrAttrBatman_left}.  Then the shape of the steady state is given by two parabola profiles on the parts only inhabited by the first species and cosine profiles where both species coexist: 
\begin{align*}
	\rho(x) =
	\left\{
	\begin{array}{ll}
		 -\dfrac{1}{2\epsilon} m_1 (x^2 -c^2) + \dfrac{m_2}{\epsilon} (x+c),	& x \in [-c,-b],	\\[0.3cm]
		 \dfrac{\hat u_2}{2}\cos\left(\dfrac{x}{\sqrt{\epsilon}}\right) - \dfrac{m_2}{2}, & x \in [-b,b],\\[0.3cm]
		 -\dfrac{1}{2\epsilon} m_1 (x^2 -c^2) - \dfrac{m_2}{\epsilon} (x-c),	& x \in [b,c],
	\end{array}
	\right.
\end{align*}
and
\begin{align*}
	\eta(x) = \frac{\hat u_2}{2}\cos\left(\frac{x}{\sqrt{\epsilon}}\right) - \frac{m_1}{2},
\end{align*}
on $[-b,b]$. Figure \ref{fig:AttrAttrBatman_right} shows an excellent agreement between numerical and analytical steady states.
\begin{figure}[!ht]
	\centering
    \subfigure[The root of Eqs.(\ref{eq:200616_1730_1}, \ref{eq:200616_1730_2}) determines the support.]{
	\label{fig:AttrAttrBatman_left}
	\includegraphics[width=0.45\textwidth]{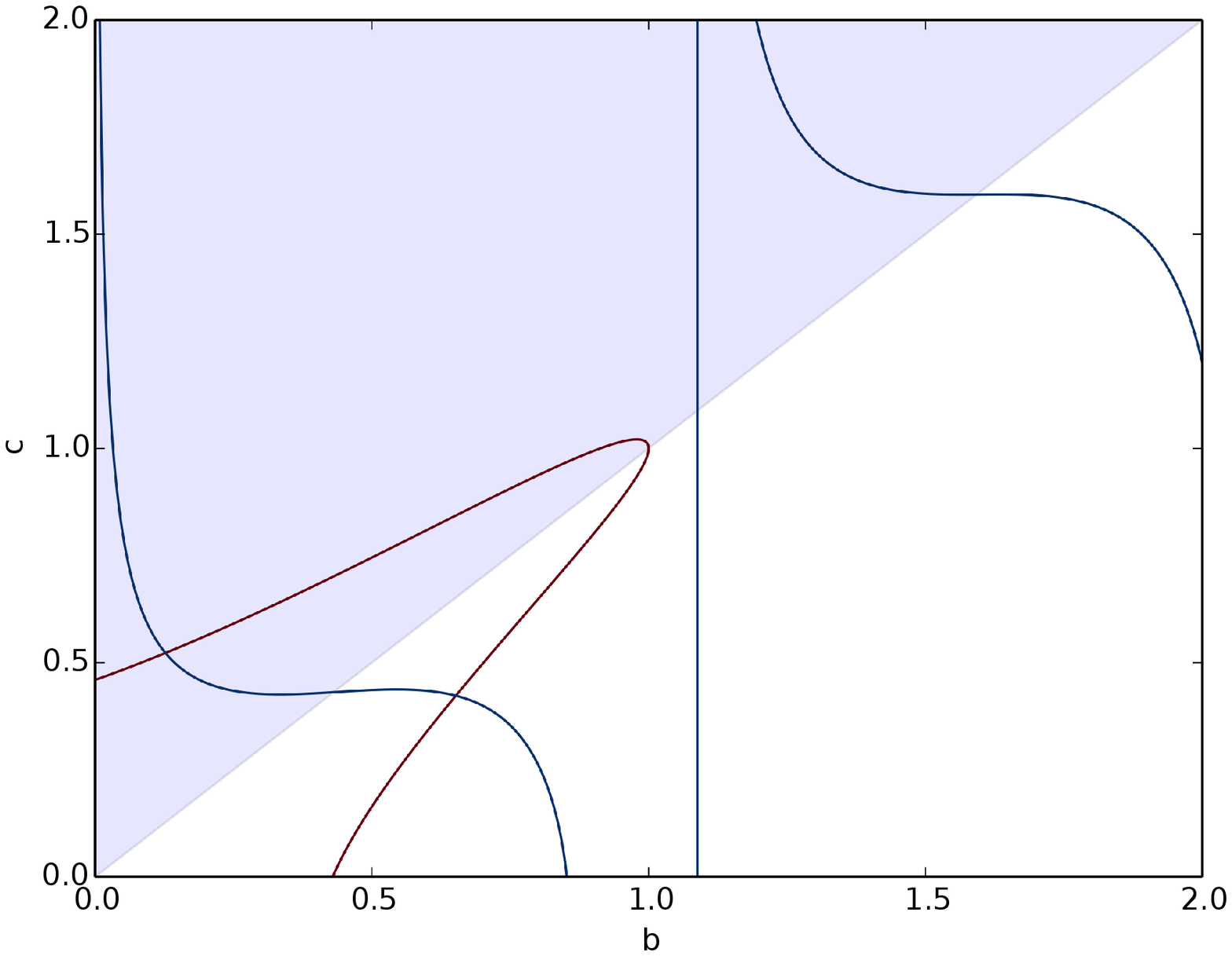}
 	}
    \hspace{0cm}
    \subfigure[$\ \epsilon = 0.12, m_1= 0.6, m_2=0.1$]{
    \label{fig:AttrAttrBatman_right}
	\includegraphics[width=0.45\textwidth]{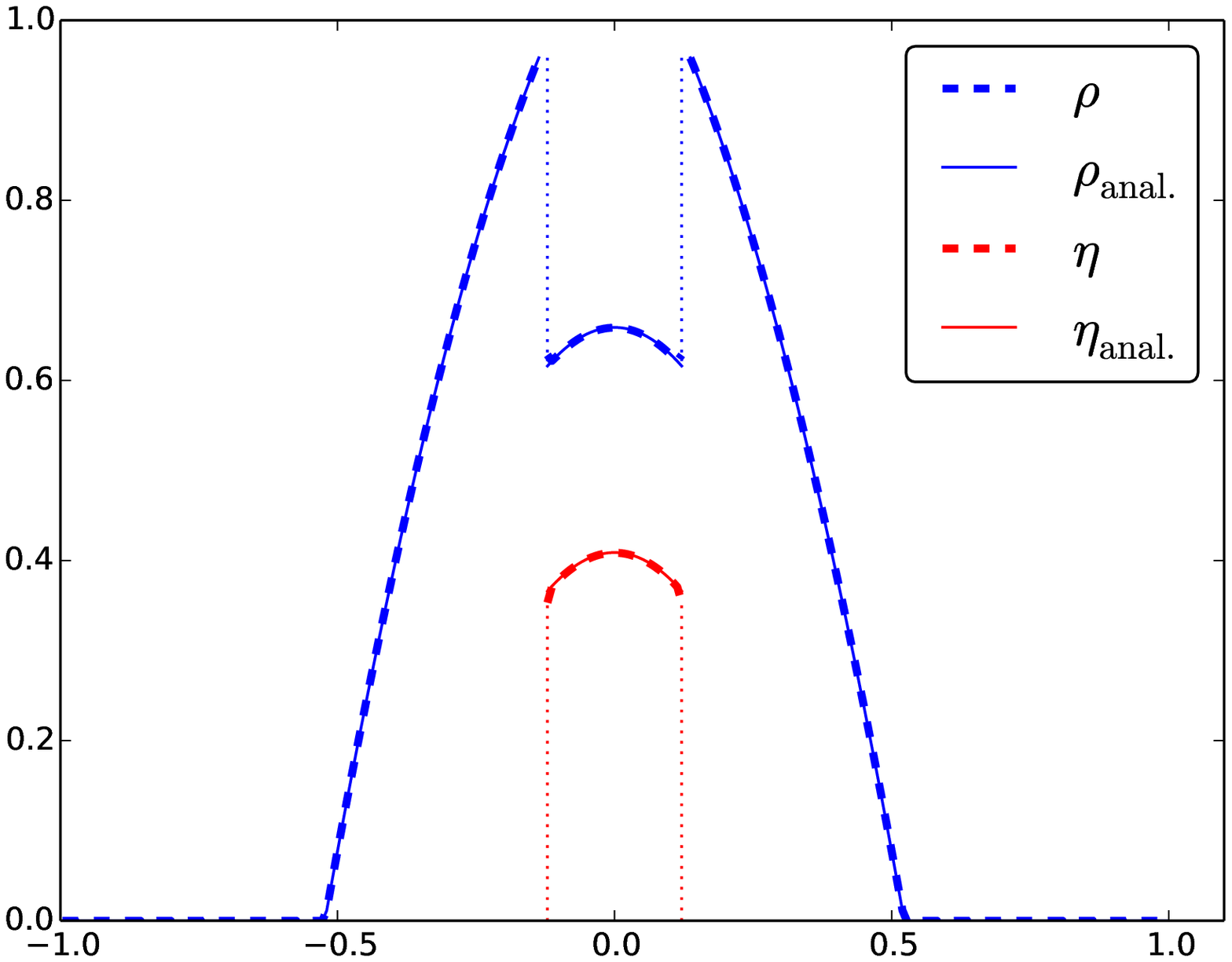}
    }
    \caption{Left: mass condition (red line) and the continuity of the sum (blue line) give rise to two equations for the support. The shaded area is the condition $c\geq b$. Right: analytical (straight lines) and numerical (dashed lines) Batman profile agree perfectly.}
    \label{fig:AttrAttrBatman}
\end{figure}
Let us remark that Eq. \eqref{eq:200616_1730_1} implies $b=c$ in the case of $m_1 = m_2$. As a consequence both species completely overlap and the profile is just that of a cosine, cf. Figure \ref{fig:AttrAttrCompleteOverlap}.
\begin{figure}[!ht]
	\centering
    \subfigure[The root of Eqs.(\ref{eq:200616_1730_1}, \ref{eq:200616_1730_2}) determines the support.]{
	\label{fig:AttrAttrCompleteOverlap_left}
	\includegraphics[width=0.45\textwidth]{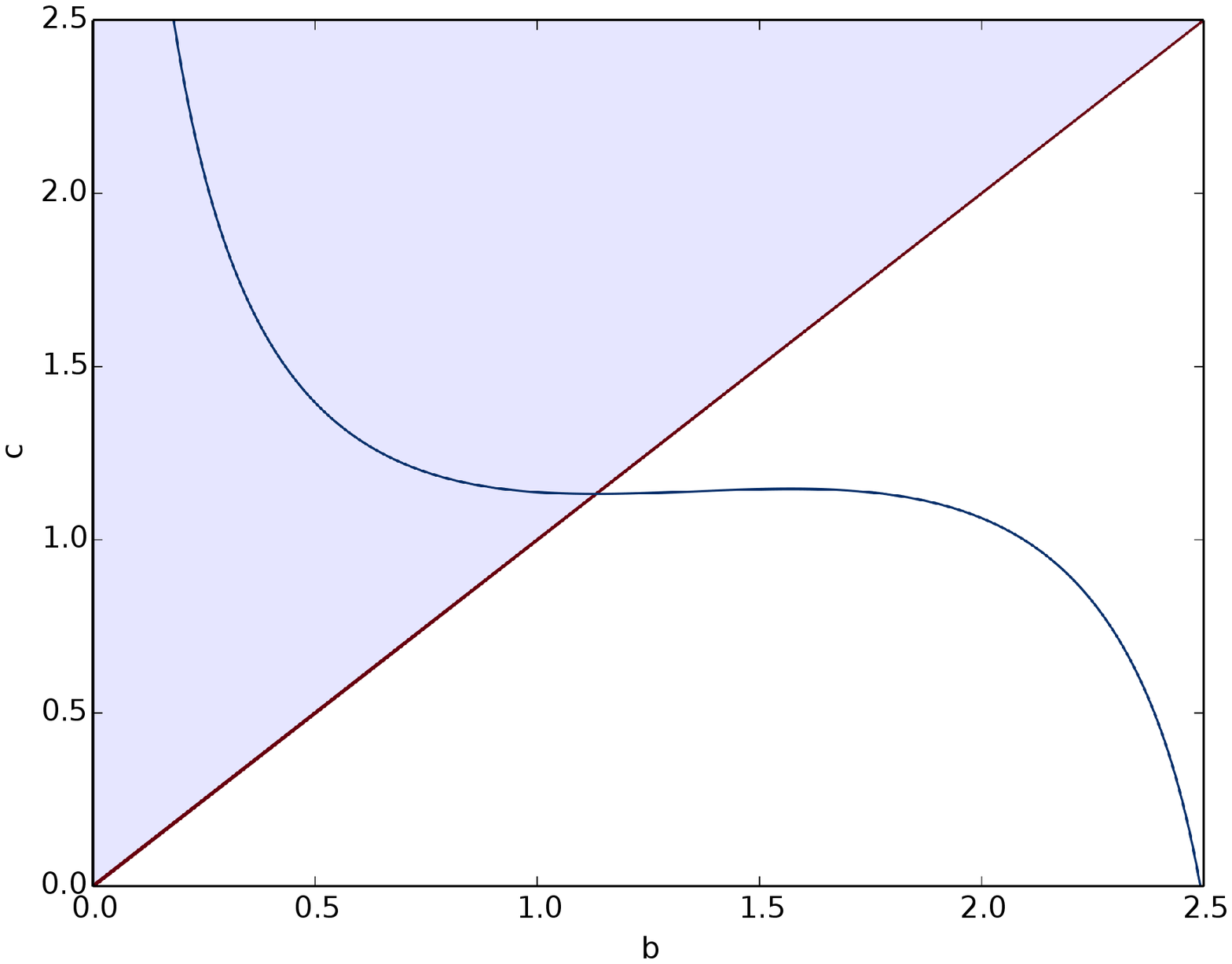}
	}
    \hspace{0cm}
    \subfigure[$\ \epsilon = 1, m_1=1, m_2=1$]{
    \label{fig:AttrAttrCompleteOverlap_right}
	\includegraphics[width=0.45\textwidth]{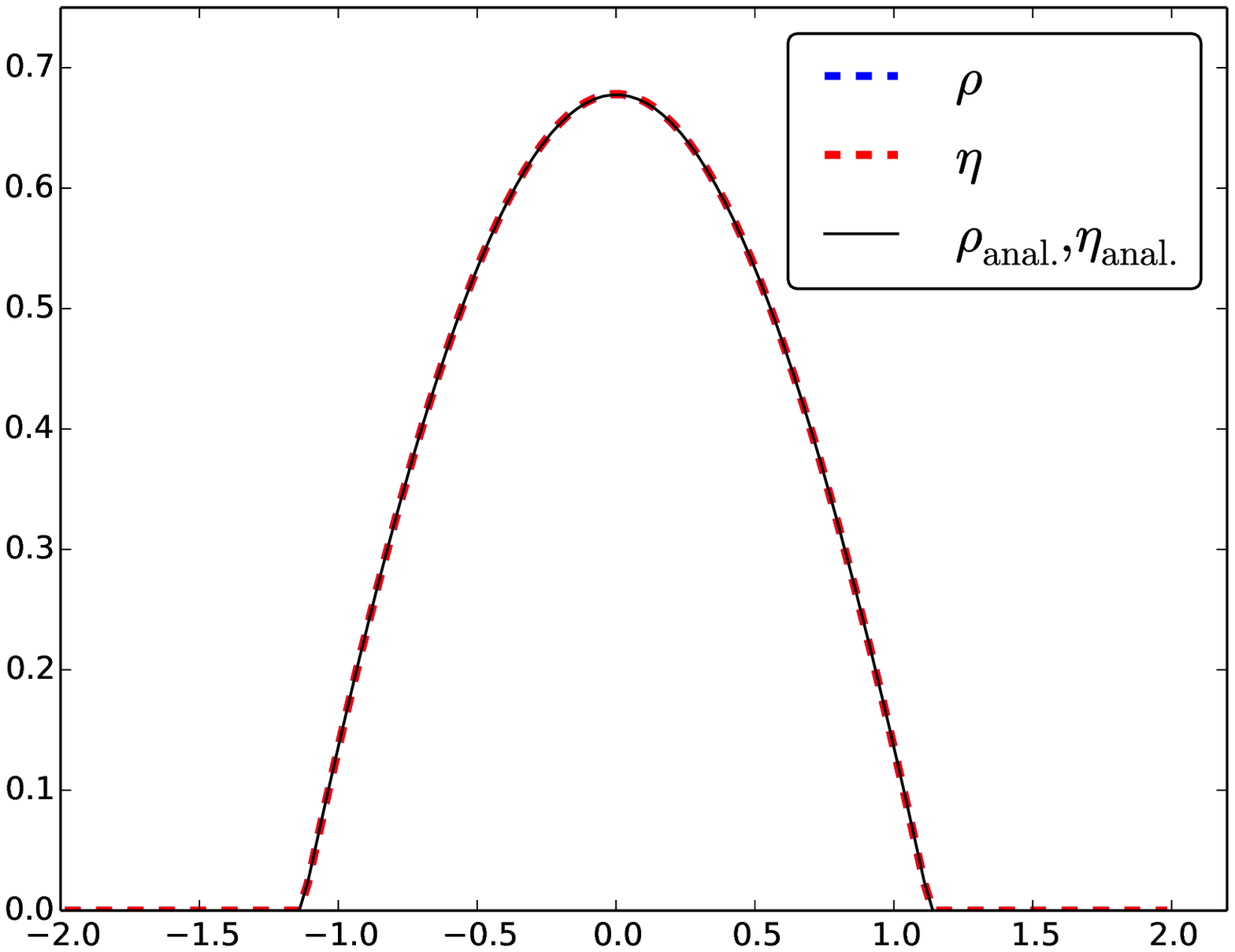}
    }
    \caption{Stationary distribution in the case of same masses. Left: mass condition and  continuity of the sum determine the support. Right: the analytical (straight lines) and the numerical (dashed lines) steady state agree perfectly.}
    \label{fig:AttrAttrCompleteOverlap}
\end{figure}
Numerical simulations show that the Batman profiles are the only symmetric stationary distribution in a certain range of cross-diffusivities, namely $(0,\epsilon^{(1)}]$. For $\epsilon \in (\epsilon^{(1)},\epsilon^{(2)}]$,  a new family of profiles (called \emph{the second kind})  emerges coexisting with the Batman profiles in this range, cf. Figure \ref{fig:secondkindprofile}. Finally, for $\epsilon> \epsilon^{(2)}$ only  profiles of the second kind prevail.
\begin{figure}[!ht]
	\centering
	\subfigure{
	\includegraphics[width=0.47\textwidth]{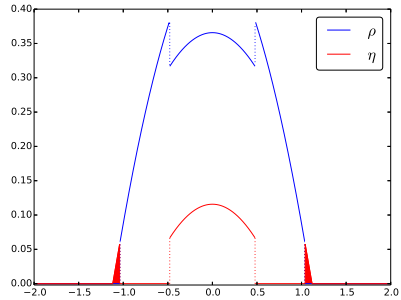}
	}
	\caption{For $\epsilon>\epsilon^{(1)}$ a second kind of profile surfaces. In fact, there is a whole family of steady states parameterised by a the mass fraction, $p\in [p_{\mathrm{min}}, p_{\mathrm{max}}]$, in the corners (filled in red).}
	\label{fig:secondkindprofile}
\end{figure}
Since the steady states are a state of balance between diffusion and attractive interactions, the second kind of profiles can be seen as states in which the attractive force is not strong enough to ensure the formation of a single group for $\eta$ as observed in the Batman profiles.

Similarly to the Batman profiles, we may determine parameters and their governing equations for profiles of second kind. In the symmetric case, using~\eqref{eq:steadystatecond} the profiles are given by
\begin{align*}
		\rho(x) =\! \left\{\!
		\begin{array}{rl}
			\rho^L(x)&=\frac1\epsilon\!\left(\frac{m_2}{2}(d^2-c^2) + m_1(d-c) +\frac{m_1}{2}(c^2-x^2) + (1-p)m_2(c+x)\right)\\
			\rho^M(x)&=\frac{B}{2} \cos\left(\frac{x}{\sqrt{\epsilon}}\right)-\frac{m_2}{2}\\
			\rho^R(x)&=\frac1\epsilon\!\left(\frac{m_2}{2}(d^2-c^2) + m_1(d-c) +\frac{m_1}{2}(c^2-x^2) + (1-p)m_2(c-x)\right)
		\end{array}
		\right.
	\end{align*}
	where $\mathrm{supp}(\rho^L)=[-c,-b], \mathrm{supp}(\rho^M)=[-b,b]$, $\mathrm{supp}(\rho^R)= [b,c]$, and $p$ is the fraction of mass in the corners of $\eta$, cf. Figure \ref{fig:secondkindprofile}, (areas filled in red). Similarly,
	\begin{align*}
		\eta(x)=\left\{\begin{array}{rl}
			\eta^L(x) &= \frac1\epsilon\big(\frac{m_2}{2}(d^2-x^2) +m_1(d+x)\big)\\
			\eta^M(x)&=\frac{B}{2}\cos\left(\frac{x}{\sqrt{\epsilon}}\right) -\frac{m_1}{2}\\
			\eta^R(x) &= \frac1\epsilon\big(\frac{m_2}{2}(d^2-x^2) +m_1(d-x)\big)
		\end{array}
		\right.
		\hfill
	\end{align*}
	where $\mathrm{supp}(\eta^L)=[-d,-c], \mathrm{supp}(\eta^M)=[-b,b]$, and $\mathrm{supp}(\eta^R)= [c,d]$.
It is apparent that there are five unknowns $b,c,d$ for the support, $B$ for the amplitude in regions of coexistence, and $p$ for the mass fraction. Correspondingly, we find four conditions in order to determine all parameters but $p$:
\begin{align*}
	pm_2  		=  2 \int_c^d \eta^R(x)\d x, \qquad \text{and} \qquad	(1-p)m_2 	=	\int_{-b}^b \eta^M(x)\d x,
\end{align*}
for the mass near the corners and on the middle interval $[-b,b]$, respectively. Similarly, we know that
\begin{align*}
	m_1 = \int_{-c}^{c} \rho(x) \d x,\qquad \text{and} \qquad \lim_{x\uparrow b} \sigma(x) =\lim_{x\downarrow b} \sigma(x) \qquad \text{and} \qquad \lim_{x\uparrow c} \sigma(x) =\lim_{x\downarrow c} \sigma(x),
\end{align*}
for the mass of $\rho$ and the continuity of the sum $\sigma=\rho+\eta$ at $x=c$ and $x=b$. 
Since $p$ parameterises a family of solutions and describes  both branches (as envelope) of the bifurcation diagram, cf. Figure \ref{fig:bifurcation}, we are interested in finding the conditions leading to $p_\mathrm{min}(\epsilon), p_\mathrm{max}(\epsilon)$ in the diagram, Figure \ref{fig:bifurcation}.
\begin{figure}[!ht]
	\centering
	\includegraphics[width= 0.55\textwidth]{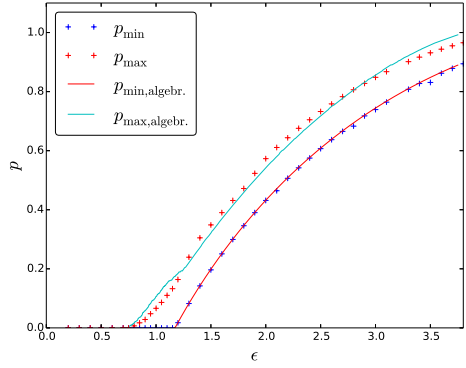}	
	\caption{Depending on the cross-diffusivity we observe different symmetric steady states. The dashed curve shows the minimal mass fraction in the corner and the dotted line the maximal mass fraction leading to a stable stationary distribution.}
	\label{fig:bifurcation}
\end{figure}
In order to determine the bifurcation diagram we run simulations with two different types of initial data -- on the one hand we start the system with $\supp(\eta)\subset \supp(\rho)$, on the other hand we initialise the system such that $\eta$ is supported around $\rho$, cf. first row of Figure \ref{fig:numerically_determining_bifurcation}. 
\begin{figure}[ht!]
	\centering
	\subfigure{
		\includegraphics[width=0.4\textwidth]{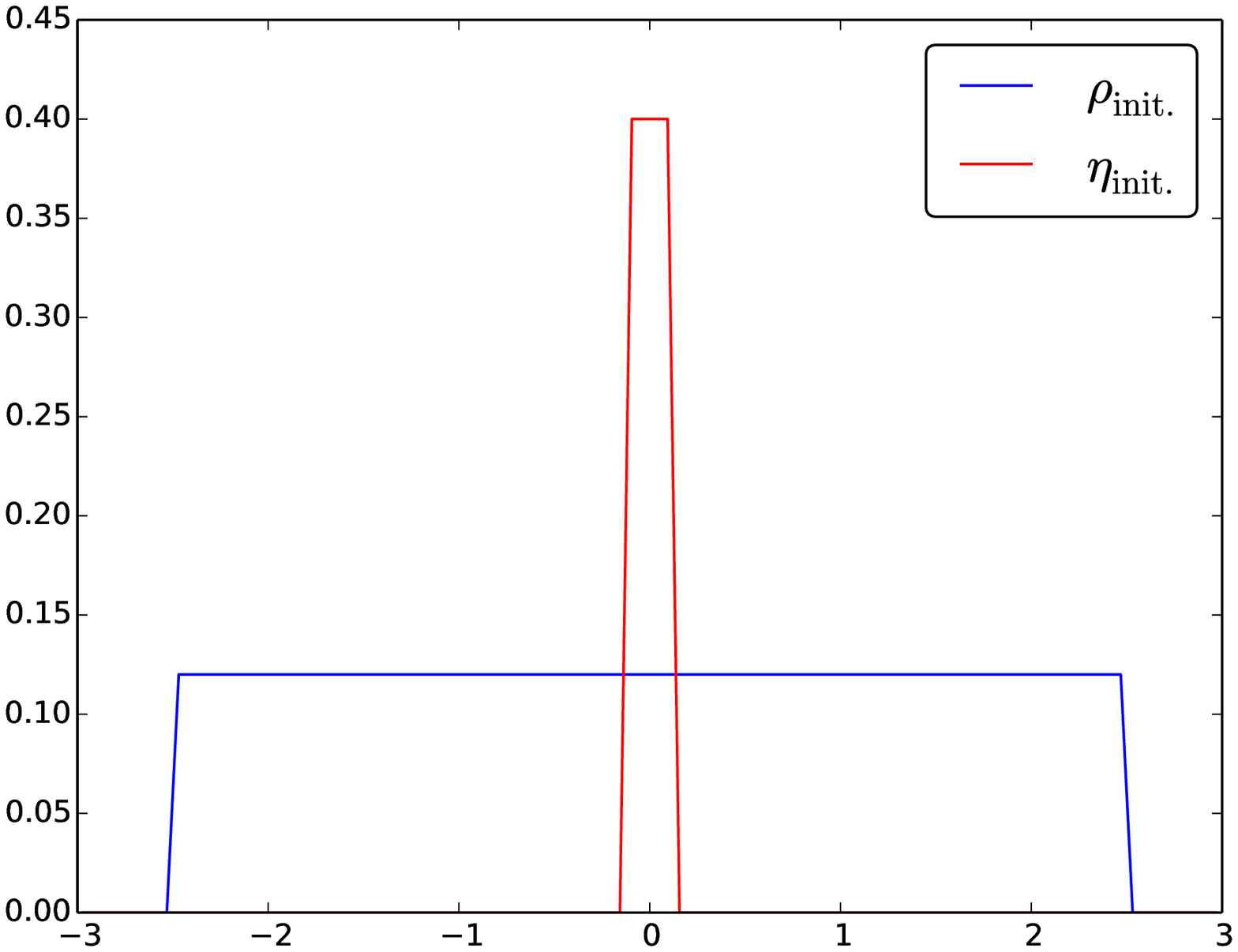}
	}
	\subfigure{
		\includegraphics[width=0.4\textwidth]{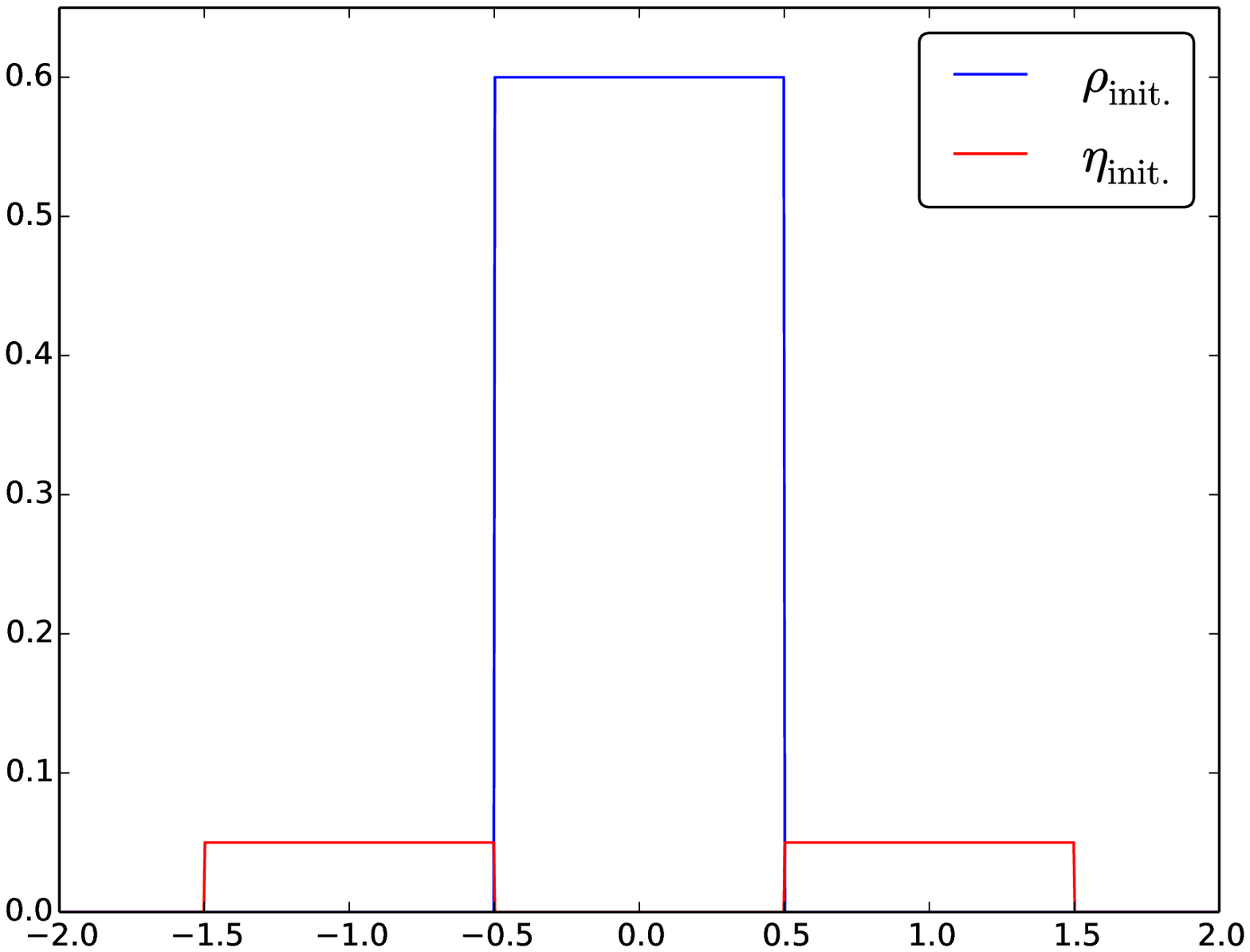}
	}
	\subfigure{
		\includegraphics[width=0.4\textwidth]{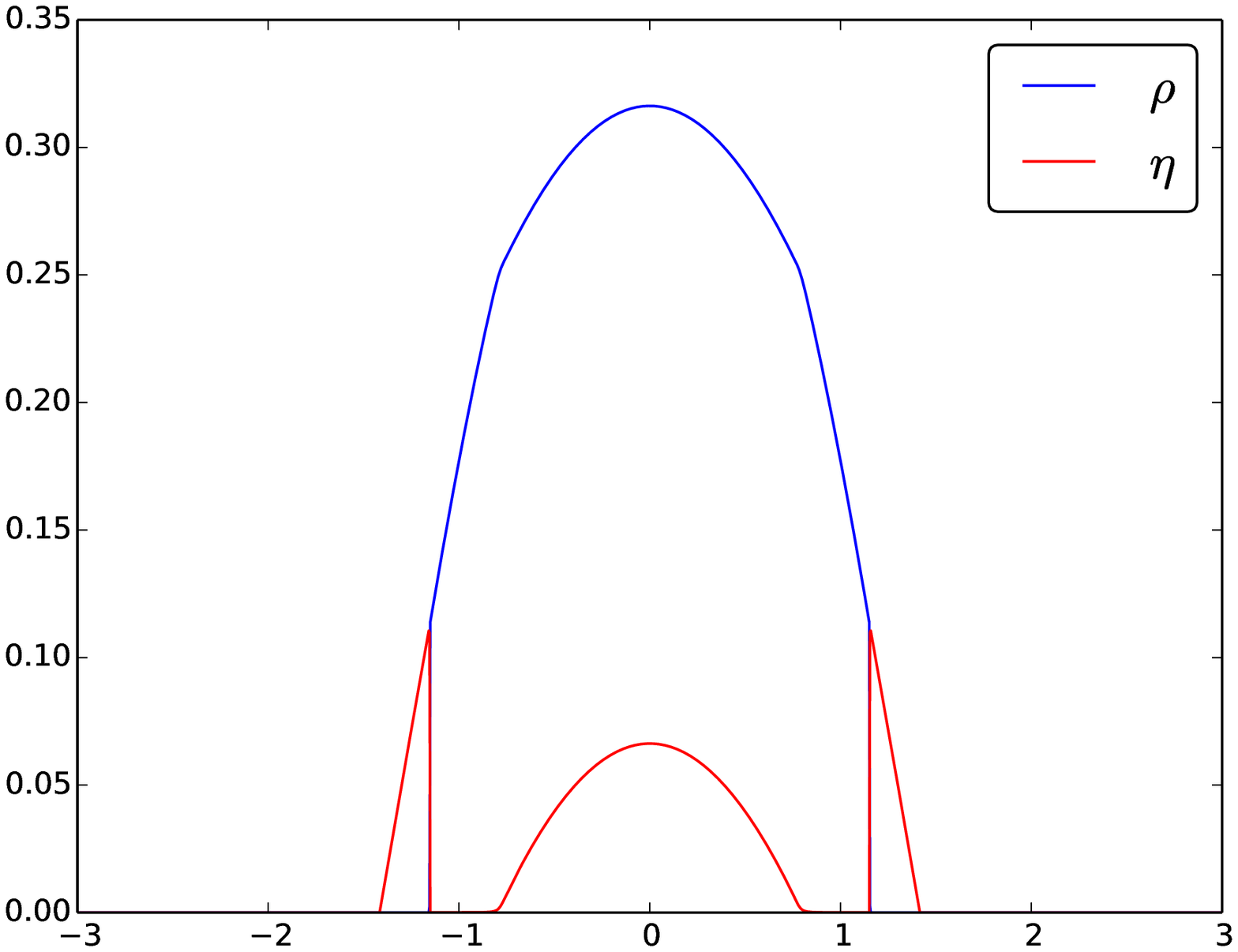}
	}
	\subfigure{
		\includegraphics[width=0.4\textwidth]{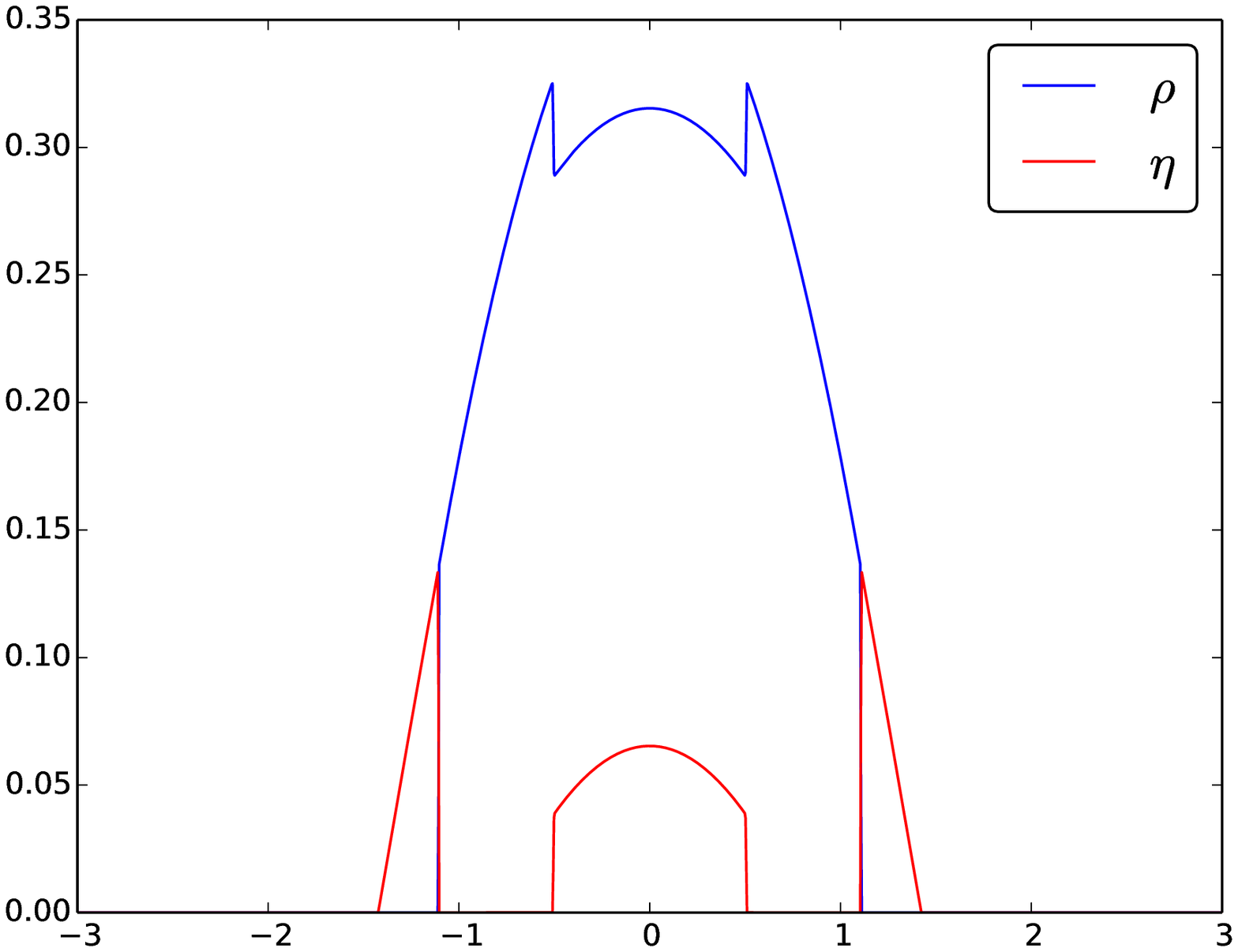}
	}
	\caption{In all four graphs the masses are $m_1 = 0.1$, $m_2=0.6$, and the cross-diffusivity is $\epsilon = 1.7$. The first row depicts two different initial data -- one (left) where $\eta$ is included in $\rho$, and one (right) where $\eta$ surrounds $\rho$. In the second row we present the corresponding steady states. Albeit having a similar make-up, they differ in their respective mass fraction of the corner, $p$. The left graph  gives the minimal mass fraction $p_\mathrm{min}$ while the right graph gives the maximal, $p_\mathrm{max}$, respectively, cf. Figure \ref{fig:bifurcation}.}
	\label{fig:numerically_determining_bifurcation}
\end{figure}
The second row shows the stationary distribution asymptotically achieved with the respective initial data. We note that the mass fraction of $\eta$ in the corners is different for both simulations albeit having used the same cross-diffusivity. The mass fraction in the left graph corresponds to $p = p_{\mathrm{min}}$ and the mass fraction in the right graph to $p=p_\mathrm{max}$, respectively. Now we want to give conditions determining the envelopes $p_\mathrm{min}(\epsilon),p_\mathrm{max}(\epsilon)$ of Figure \ref{fig:bifurcation}.

Let us impose non-negativity of $\eta$ at $x=b$, \textit{i.e.} $\eta(b) \geq 0$. This is a reasonable assumption which is also reflected in the numerical simulations, cf. Figure \ref{fig:conditions_a}. The figure shows steady states corresponding to the left initial data in Figure \ref{fig:numerically_determining_bifurcation} as $\epsilon$ increases. While we observe a discontinuity of $\eta$ at $x=b$ for small $\epsilon$, there is a critical value where $\eta(b)=0$, for all $\epsilon > \epsilon^{(1)}$.
\begin{figure}[ht!]
	\subfigure[$\eta(b)= 0$.]{
		\includegraphics[width=0.47\textwidth]{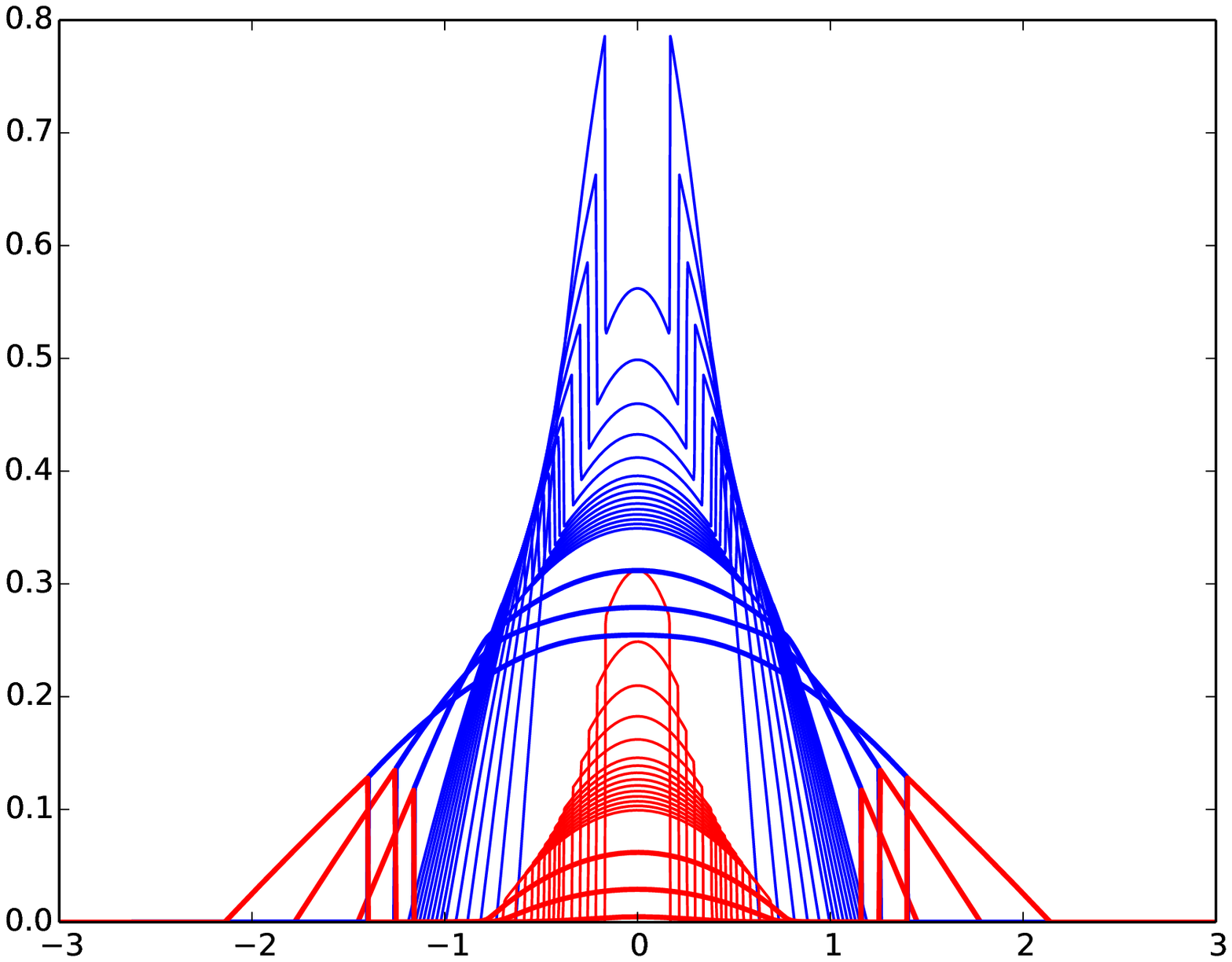}
		\label{fig:conditions_a}
	}
	\subfigure[$u_2(c)>0$.]{
		\includegraphics[width=0.47\textwidth]{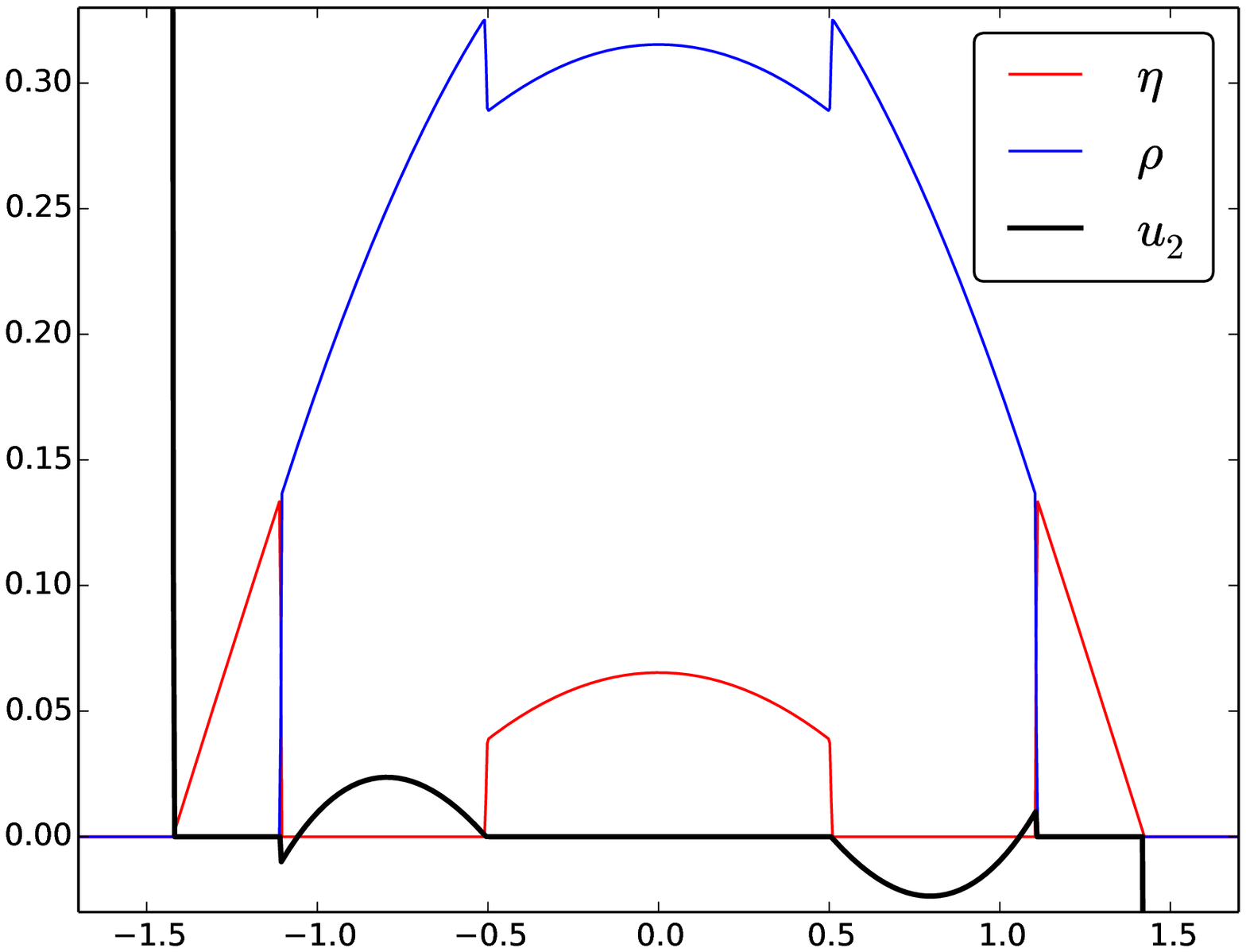}
		\label{fig:conditions_b}
	}
	\caption{Conditions for lower and upper boundary of the bifurcation diagram.}
	\label{fig:conditions}
\end{figure}
For the upper envelope we impose that the velocity field $u_2$ is non-negative  at $x=c$ since otherwise any small perturbation will render the stationary state unstable, \textit{i.e.} mass would get transported into the interior, cf. Figure \ref{fig:conditions_b}. These two conditions describe both envelopes in Figure \ref{fig:bifurcation}.

\paragraph{Vanishing diffusion regime}
In this section we study the case of Batman profiles as $\epsilon \rightarrow 0$. Recall the two equations for $b$ and $c$,
\begin{align}
  (c-b)^2(3m_2+2cm_1+m_1b)+3\epsilon(m_2-m_1+m_1b-m_2b) &= 0, \label{eq:supp1}\\
  m_1(c^2-b^2)+\epsilon(m_1+m_2)+2m_2(c-b)
  -2\sqrt{\epsilon}(m_2+m_1b)\cot\frac{b}{\sqrt{\epsilon}} &= 0. \label{eq:supp2}
\end{align}
When $\epsilon$ is small, both $b$ and $c$ are  $O(\sqrt{\epsilon})$,
suggesting
\[
  b = \epsilon^{1/2}\left( b_0 + \epsilon^{1/2} b_1 + \epsilon b_2 +
  \cdots\right),\qand
  c = \epsilon^{1/2}\left( c_0 + \epsilon^{1/2} c_1 + \epsilon c_2 +
  \cdots\right).
\]
Upon substitution of the asymptotic expansions into Eq. \eqref{eq:supp1} and \eqref{eq:supp2}, the leading order coefficients $b_0$
and $c_0$ satisfy
\begin{align*}
  m_2-m_1+(c_0-b_0)^2m_2=0,\qand b_0-c_0+\cos b_0 =0,
\end{align*}
whence
\[
  b_0 = \arccos \sqrt{\frac{m_1-m_2}{m_2}},\quad \text{and} \quad
  c_0 =  \sqrt{\frac{m_1-m_2}{m_2}}+
  \arccos \sqrt{\frac{m_1-m_2}{m_2}}.
\]
Notice that both densities in the Batman profiles will converge to a Dirac Delta at zero with the respective masses while keeping their shape with this described asymptotic scaling for their supports.

\paragraph{Asymmetric profiles}
So far we only discussed symmetric steady states. However, there is an equally rich variety of non-symmetric stationary states, cf. Figure \ref{fig:nonsymmetric_steadystates_1} and Figure \ref{fig:nonsymmetric_steadystates_2}.
\begin{figure}[ht!]
	\center
	\subfigure[Asymmetric profile for $m_1 = 1$, $m_2 = 2$, and $\epsilon=3$.]
	{
		\includegraphics[width=0.4\textwidth]{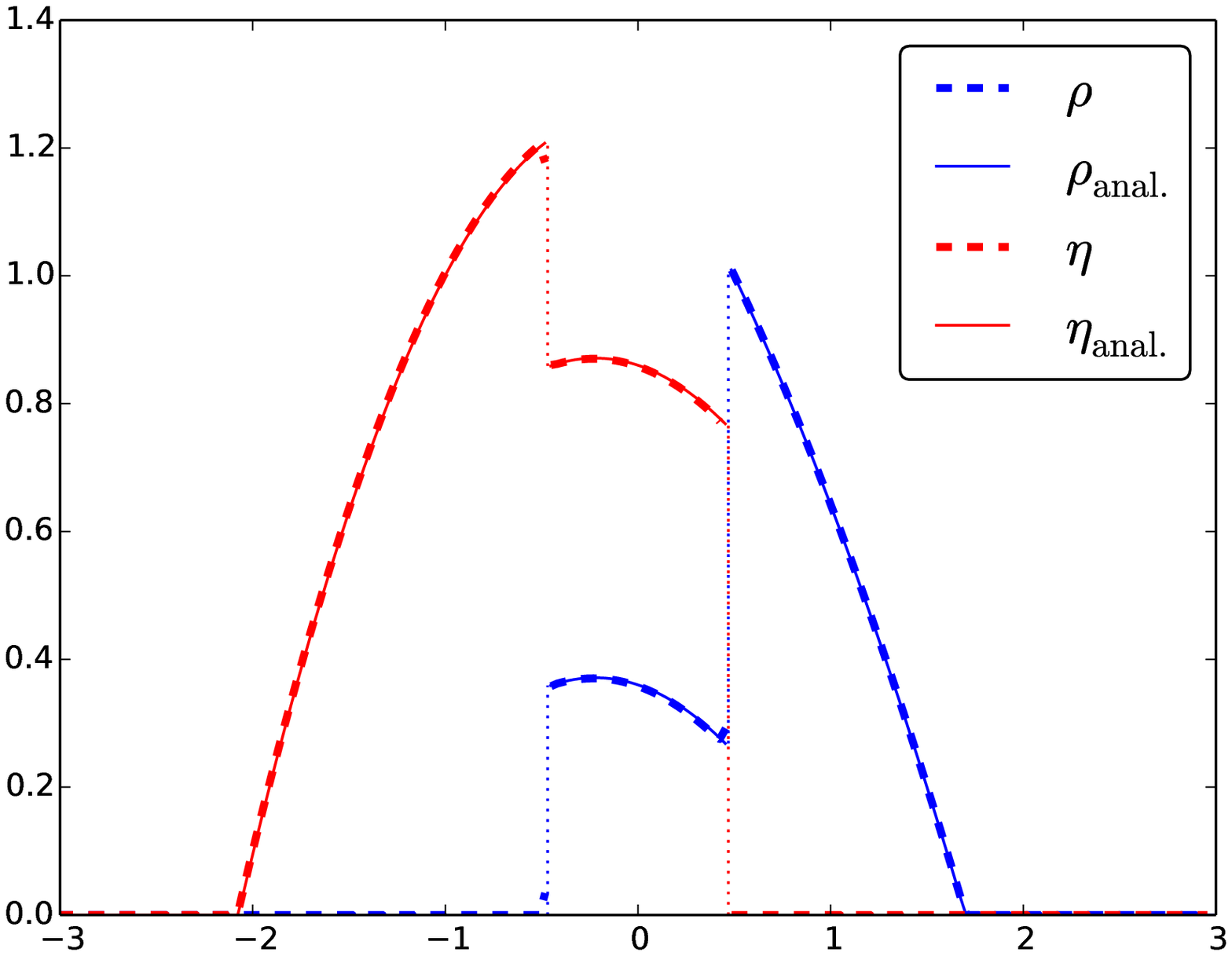}
	}
	\subfigure[Antisymmetric profile for $m_1 = m_2 = 1$ and $\epsilon=3$.]
	{
		\includegraphics[width=0.4\textwidth]{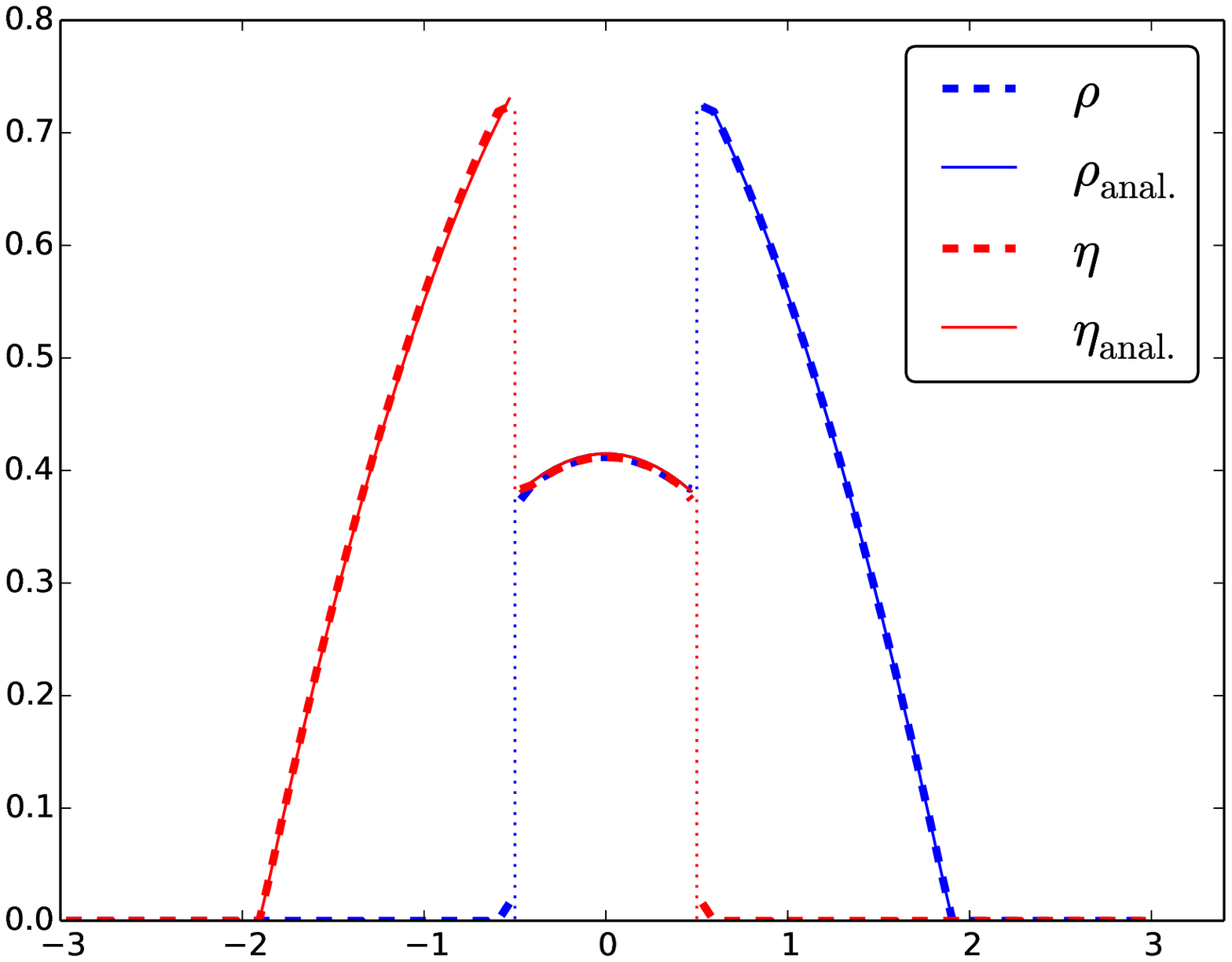}
	}
	\caption{Non-symmetric profiles for different masses (left) and equal masses (right), respectively.}
	\label{fig:nonsymmetric_steadystates_1}
\end{figure}
 In Figure \ref{fig:nonsymmetric_steadystates_1} we display the cases where the support only consists of three pieces -- two regions inhabited by only one species and the middle one where both species coexist. Figure \ref{fig:nonsymmetric_steadystates_2} on the other hand shows three further examples of asymmetric steady states suggesting the existence of an infinite family of solution of asymmetric steady states. All states have the same qualitative profile in common but differ in their respective supports and mass distribution.
\begin{figure}[ht!]
	\centering
	\subfigure
	{
		\includegraphics[width=0.3\textwidth]{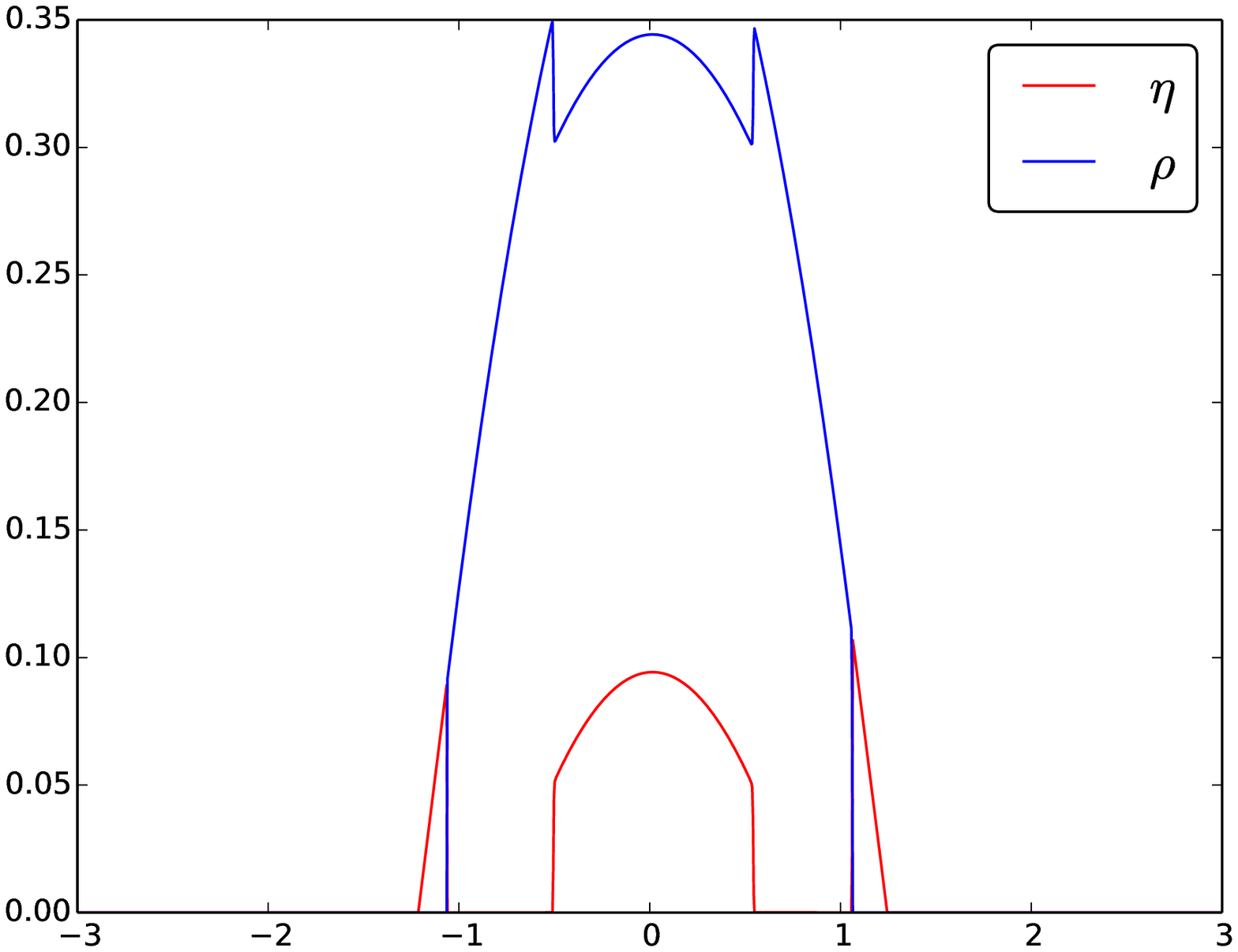}
	}
	\hspace{0.02cm}
	\subfigure
	{
		\includegraphics[width=0.3\textwidth]{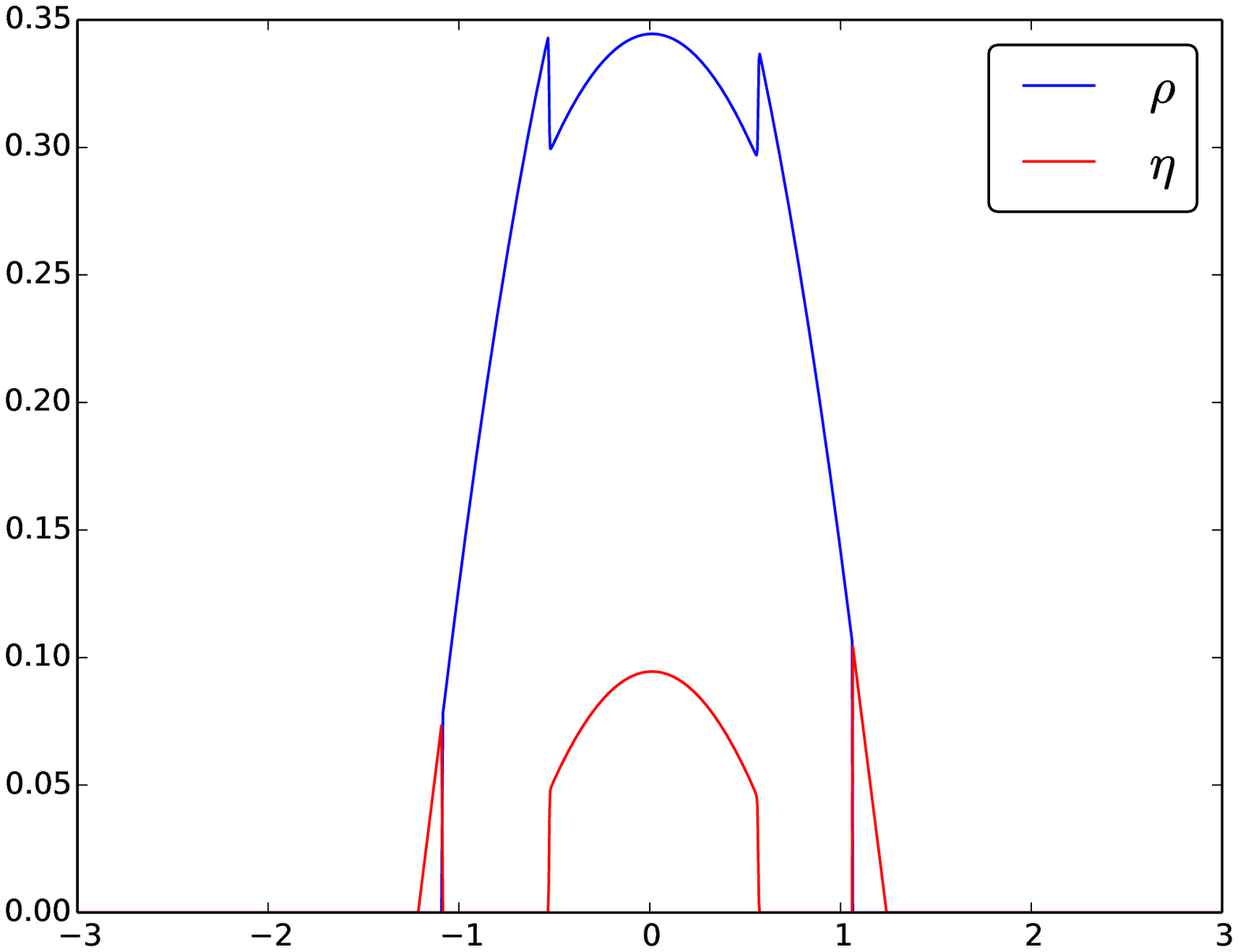}
	}
	\hspace{0.02cm}
	\subfigure
	{
		\includegraphics[width=0.3\textwidth]{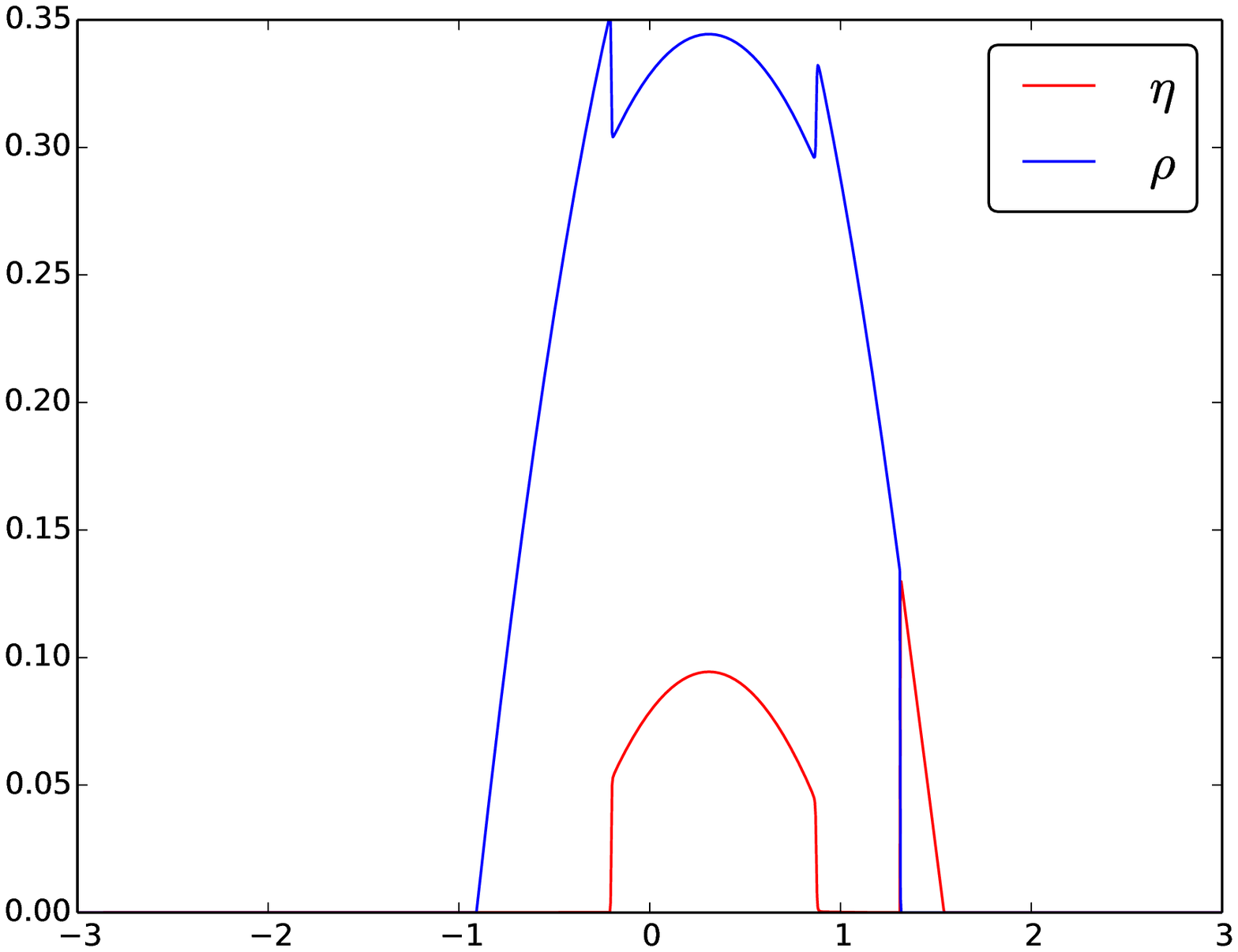}
	}
	\caption{There is an infinite family of asymmetric steady states depending only on the initial data. Here $m_1=0.6$, $m_2=0.1$, and $\epsilon=1.2$ The mass in the left corner decreases from left to right.}
	\label{fig:nonsymmetric_steadystates_2}
\end{figure}
We do not observe asymmetric profiles for $0<\epsilon<\epsilon^{(1)}$ independent of the masses $m_1$ and $m_2$. Only for larger cross-diffusivities, $\epsilon>\epsilon^{(1)}$, asymmetric profiles can be observed. Moreover, there is a whole family of asymmetric profiles as can be seen in Figure \ref{fig:nonsymmetric_steadystates_2}. This is similar to the case of symmetric stationary states, parameterised by the mass fraction $p$.

\paragraph{Stability of steady states and symmetrising effect}
Let us now discuss the numerical stability of the symmetric steady states. Here the bifurcation point $\epsilon^{(1)}$ plays an important role, for the system exhibits a symmetrising effect whenever the cross-diffusivity lies below the critical one, in the sense that there is only one symmetric steady state attracting any initial data.

We fixed $\epsilon \in (0,\epsilon^{(1)})$ and chose $\rho_0=2m_1\mathbbm{1}_{[-0.5,0]}$ and $\eta_0=2m_2\mathbbm{1}_{[0,0.5]}$ for all combinations of  masses  of the form $(m_1,m_2) = 0.1\cdot\,(i,j)$ for $i,j=1\ldots 10$. In all cases we observe that there is only one attractor, namely the Batman profile of the form given in Figure \ref{fig:AttrAttrBatman_right} and Figure \ref{fig:AttrAttrCompleteOverlap_right} in the case $m_1=m_2$, respectively. For $\epsilon>\epsilon^{(1)}$ the system is not symmetrising anymore and small perturbations lead to different stationary states. This can be seen if $p$ is varied in $[p_\mathrm{min}, p_\mathrm{max}]$, for it leads to different states. A similar argument holds for the asymmetric states, by shifting mass from one corner into the other, cf. Figure \ref{fig:nonsymmetric_steadystates_2}.

\subsection{Attractive-repulsive case}
In this section we present the attractive - repulsive case, \textit{i.e.} $W_{12}= |x| = -W_{21}$. Then the steady states have segregated densities, as asserted by the following proposition.
\begin{proposition}[Completely segregated steady states.]
\label{prop:segregation}
Let $(\rho, \eta)$ be a stationary solution of system \eqref{eq:our_system}. Then both species do not intermingle, \textit{i.e.} there cannot be connected components of $\supp(\rho) \cap \supp(\eta)$	with non empty interior.
\end{proposition}
\begin{proof}
Suppose the interior of a connected componente of $\supp(\rho)\cap\supp(\eta)$ is not empty. We know that both species satisfies Eqs. \eqref{eq:steadystatecond} in that connected component:
\begin{align*}
	\begin{split}
	c_1 &= W_{11}\star \rho + W_{12}\star\eta + \epsilon(\rho + \eta),\\
	c_2 &= W_{22}\star \eta + W_{21}\star\rho + \epsilon(\rho + \eta).
	\end{split}
\end{align*}
Similar arguments as above imply that the interaction terms are twice differentiable in this interval, thus we differentiate twice and get
\begin{align}
\label{eq:150816_1722}
	0 = m_1 + 2\eta + \epsilon (\rho + \eta)'' \qand 0 = m_2 - 2\rho + \epsilon (\rho + \eta)''.
\end{align}
Upon subtracting both equations we deduce $0 = m_1 - m_2 + 2(\rho+ \eta)$, or equivalently
$\rho + \eta = \frac{m_2-m_1}{2}$.
But then Eq. \eqref{eq:150816_1722} reduces to
$\eta =-m_1$ and $\rho =m_2$.
Clearly this is a contradiction to the non-negativity of the densities: $0\leq \eta =-m_1<0$.
Thus the species do not intermingle.
\end{proof}
\subsubsection{Steady states}
This section is dedicated to studying the steady states of the system with attractive-repulsive cross-interactions. Due to numerical simulations and the previous proposition we make the following assumption on the support
\begin{align*}
	\supp(\rho) = [-c,c], \qquad \text{and} \qquad \supp(\eta) = [a,b] \cup [d,e],
\end{align*}
where $a<b \leq -c <c \leq d <e$ are some real numbers. Using Eqs. \eqref{eq:steadystatecond} we proceed similar as above, cf. Eqs. (\ref{eq:SelfInteractionTerm}, \ref{eq:CrossInteractionTerm}), to obtain
\begin{align*}
	\eta^L(x) = \frac1\epsilon \left(c_2 - \frac12 m_2 x^2 + M_2 x -\frac12 \bar M_2  + M_1 -m_1 x\right),
\end{align*}
for shape of the second species on the left part of the support and
\begin{align*}
	\eta^R(x) = \frac1\epsilon \left(c_2 - \frac12 m_2 x^2 + M_2 x -\frac12 \bar M_2  - M_1 + m_1 x\right),
\end{align*}
for the right part, respectively. Similar as above, we can see that the interaction terms are twice differentiable, therefore differentiating Eq. \eqref{eq:steadystatecond} in the support of $\rho$ twice yields $0 = m_1 + \epsilon(\rho)''$, and thus
\begin{align*}
	\rho(x) = -\frac{1}{2\epsilon}m_1 x^2 + \beta x + \gamma,
\end{align*}
with  $\beta,\gamma$ to be determined. Again we impose the continuity of the sum at the boundary points of each part of the supports, \textit{i.e.}
\begin{align}
	\label{eq:continuitycond}
	\eta^L(a) = 0, \quad \eta^L(b) = \rho(-c), \quad \rho(c) = \eta^R(d), \quad \text{and} \quad \eta^R(e) = 0,
\end{align}
where $\rho(-c) = 0$ if $b<-c$, and $\rho(c) = 0$ if $c<d$. We compute
\begin{align*}
\begin{split}
	\eta^L(x) &= \eta^L(x) - \eta^L(a) =-\frac{1}{2\epsilon} m_2 (x^2 -a^2) + \frac{M_2 - m_1}{\epsilon} (x-a),
\end{split}
\end{align*}
and analogously
\begin{align*}
\begin{split}
	\eta^R(x) &= \eta^R(x) - \eta^R(e) =-\frac{1}{2\epsilon} m_2 (x^2 -e^2) + \frac{M_2 + m_1}{\epsilon} (x-e).
\end{split}
\end{align*}
Concerning the first species, the parameters $\beta, \gamma$ are determined by the continuity condition Eq. \eqref{eq:continuitycond} and we obtain
\begin{align*}
	\rho(x) = -\frac{1}{2\epsilon}m_1 (x^2-c^2) + \frac{\eta^R(d) - \eta^L(b)}{2c} x + \frac{\eta^R(d) + \eta^L(b)}{2}.
\end{align*}
We can see that there are six unknowns, namely $a,b,c,d,e, \text{and }M_2$ with a total of five conditions:
\begin{align*}
\begin{gathered}
	\int \rho \d x = m_1, \quad \int \eta^L \d x  = \frac12 m_2, \quad \int \eta^R\d x= \frac12 m_2,\\
	\int x \rho \d x = M_1, \quad \int x \eta \d x = M_2,
\end{gathered}
\end{align*}
by imposing half of the mass of $\eta$ to each side of $\rho$.

\subsubsection{Case of strict segregation}
Let us start by discussing the case
\begin{align*}
	\eta^L(b) = \rho(\pm c) = \eta^R(d) =0.
\end{align*}
Then the condition on the mass yields
\begin{subequations}
\label{eq:support_of_segregated_attrrep_steadystates}
\begin{align*}
	\int_{-c}^c \rho(x) \d x = m_1 \qquad \Rightarrow \qquad c = \sqrt[3]{\frac32 \epsilon}.
\end{align*}
We can solve $\eta^L(b) = 0$ for $a$,
\begin{align}
	\label{eq:2408216_1229}
	a= {\frac {-{m_2}\,b+2\,{M_2}-2\,{m_1}}{{m_2}}}.
\end{align}
Since half of the mass is located to the left of the first species, we get
\begin{align}
	\label{eq:30082016_0940}
	\int_{a}^{b}\eta^L(x) \d x=\frac{m_2}{2} \quad \Rightarrow \quad b={\frac {\frac12\,\sqrt [3]{6\epsilon}{m_2}+{M_2}-{m_1}}{{m_2}}},
\end{align}
where we used Eq. \eqref{eq:2408216_1229}. Similarly, we solve $\eta^R(d) = 0$ for $e$ to obtain
\begin{align}
	\label{eq:240816_1425}
	e= {\frac {-{m_2}\,d+2\,{M_2}+2\,{m_1}}{{m_2}}}.
\end{align}
Using this expression we compute
\begin{align}
	\label{eq:15092016_2014}
	\int_{d}^{e} {\eta^R}(x) \,\d x=\frac{m_2}{2} \quad \Rightarrow \quad d = {\frac {-\frac12\,\sqrt [3]{6\,\epsilon}{m_2}+{M_2}+{m_1}}{{m_2}}}.
\end{align}
\end{subequations}
So we have determined $c,b,d$ depending only on the masses and the second order moments of the second species, $M_2$. We can substitute the values into Eqs. (\ref{eq:2408216_1229}, \ref{eq:240816_1425}) to determine $a$ and $e$.

\paragraph{Critical $\epsilon$ and maximal $M_2$}
	We are interested in a condition determining as to when segregation of species occurs. In fact there is a critical value of the cross-diffusivity, $\epsilon_c$, such that there only exist adjacent steady states for $\epsilon > \epsilon_c$. For $0< \epsilon <\epsilon_c$ strictly segregated steady states occur if $|M_2| < M_{2,\rm{max}}$, where $M_2 = 0$ corresponds to the symmetric case. Figure \ref{fig:CritEpsMaxM2} displays this behaviour.
	
	Let us derive an expression for $\epsilon_c$ and $M_{2,\rm{max}}$. For a fixed $\epsilon$ we may compute $M_{2,\rm{max}}$.
	We begin with the case $c=-b$. We can solve for the critical $M_2$, \textit{i.e.}
	\begin{align}
		\label{eq:16092016_1845}
		M_{2,\rm{max}} = -\frac12m_2\sqrt[3]{\epsilon}(\sqrt[3]{12}+\sqrt [3]{6})+{m_1}.
	\end{align}
	Similarly, we can solve equation $c=d$ for $M_2$, which gives
	\begin{align}
		\label{eq:16092016_1845_1}
		\bar M_2 = \frac12m_2\sqrt[3]{\epsilon}(\sqrt[3]{12}+\sqrt [3]{6})-{m_1} = -M_{2,\rm{max}}.
	\end{align}
	Thus, the parameter $M_2$ can vary in the range $[-M_{2,\rm{max}}, M_{2,\rm{max}}]$. Then the critical value of $\epsilon$ makes this interval degenerate, \textit{i.e.} it is given by the condition $M_{2,\rm{max}} = 0$. This way we obtain
	\begin{align*}
		\epsilon_c = \frac49\,{\frac {m_1^{3} \left( \sqrt [3]{2}-1 \right) }{m_2^{3}}}.
	\end{align*}
	\begin{figure}[!ht]
	\centering
	\subfigure[\,Critical $\epsilon$.]{
	\includegraphics[width=0.35\textwidth]{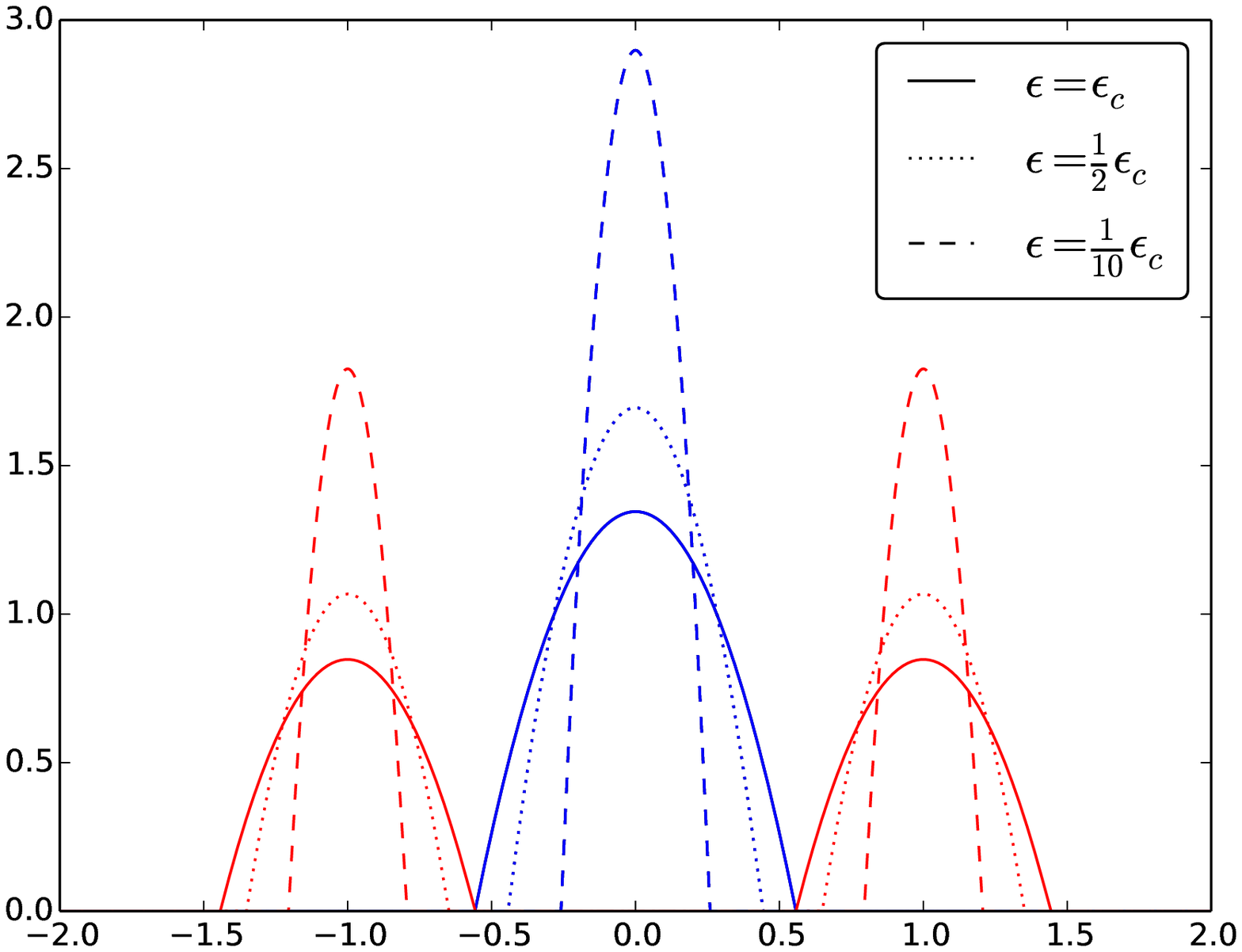}
	}
    \subfigure[\,Maximal $M_2$.]{
    \includegraphics[width=0.35\textwidth]{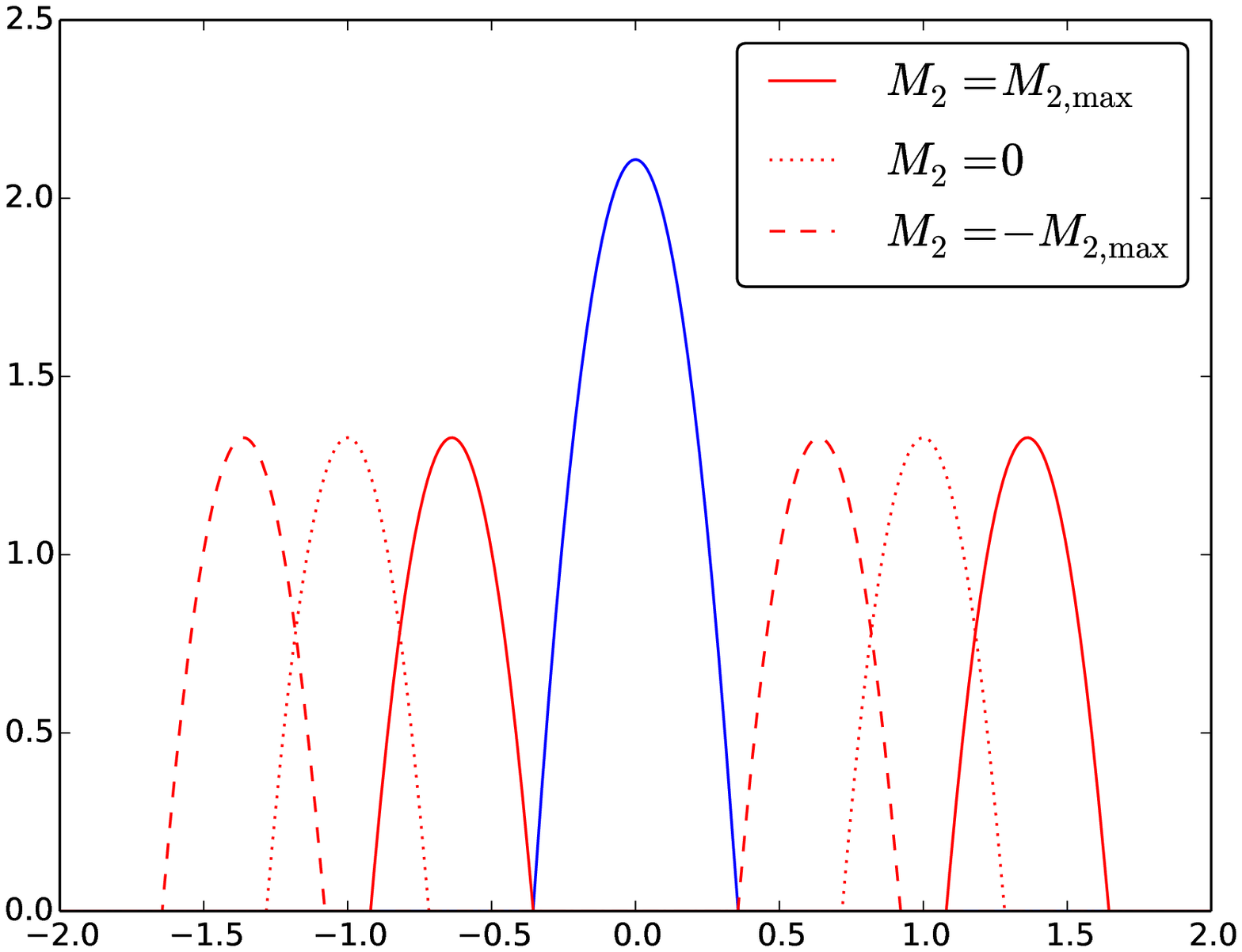}
    }
    \caption{Left: If $\epsilon\in(0, \epsilon_c)$ strictly segregated distributions are possible (dashed/dotted). If $\epsilon = \epsilon_c$ both species touch (straight line), and for $\epsilon>\epsilon_c$ strictly segregated states are no longer possible. Right: In the case $\epsilon\in (0,\epsilon_c)$ there exists a whole spectrum of steady states, parametrised by $M_2$ ranging from $-M_{2,\rm{max}}$, (dashed) to  $M_{2,\rm{max}}$ (straight line). The case $M_2 = 0$ corresponds to the symmetric case.}
    \label{fig:CritEpsMaxM2}
\end{figure}

If $\epsilon = \epsilon_c$ both species touch at the points $\{-c,c\}$ or are partially adjacent. If $\epsilon< \epsilon_c$ but we choose $M_2$ outside of the aforementioned range we observe steady states consisting of (partially) adjacent bumps. 
\begin{figure}[ht!]
	\centering
	\subfigure[$\epsilon = 1/20$.]{
		\includegraphics[width=0.31\textwidth]{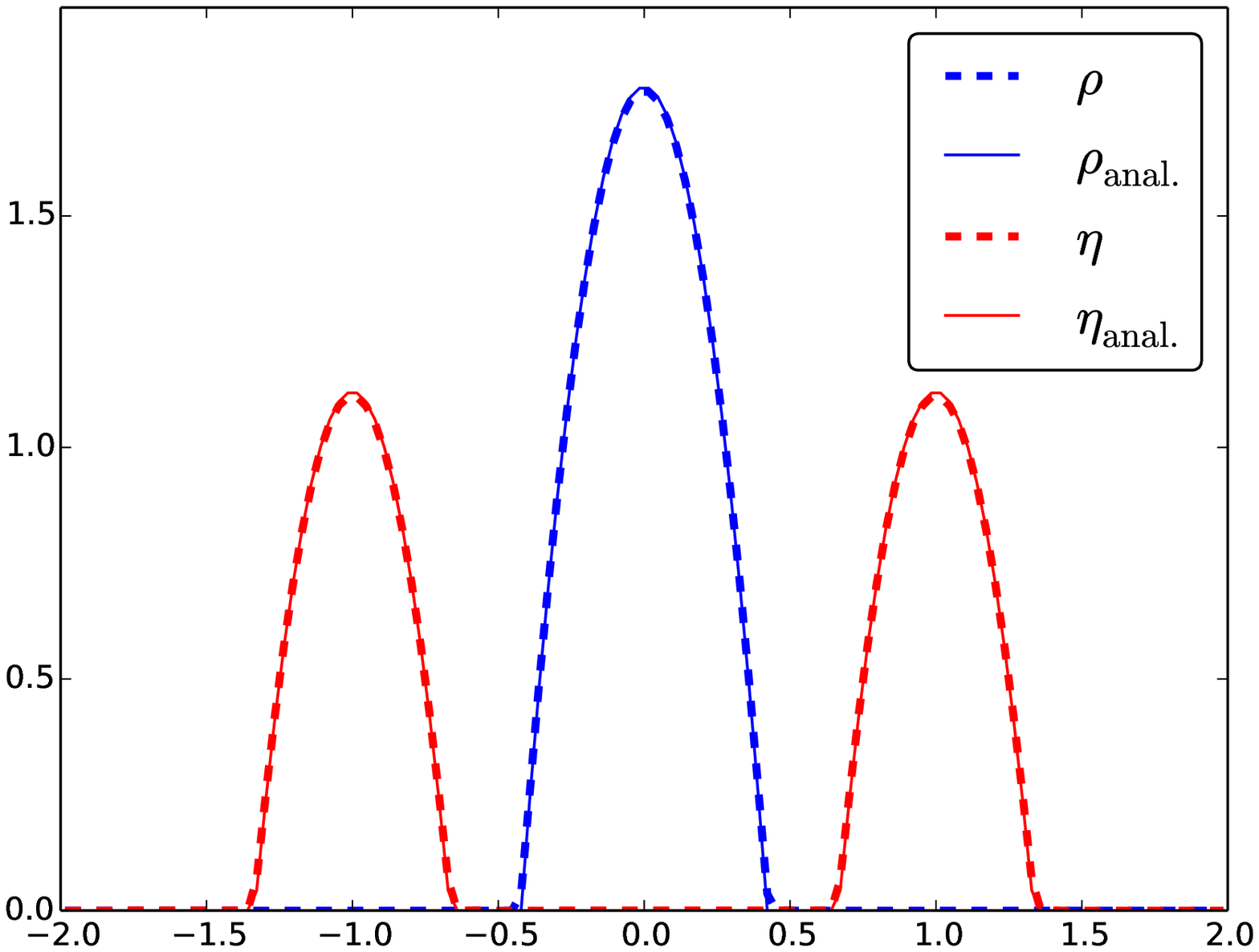}
	}
	\hspace{-1em}
	\subfigure[$\epsilon=  \frac49  (2^{1/3} - 1)$.]{
		\includegraphics[width=0.31\textwidth]{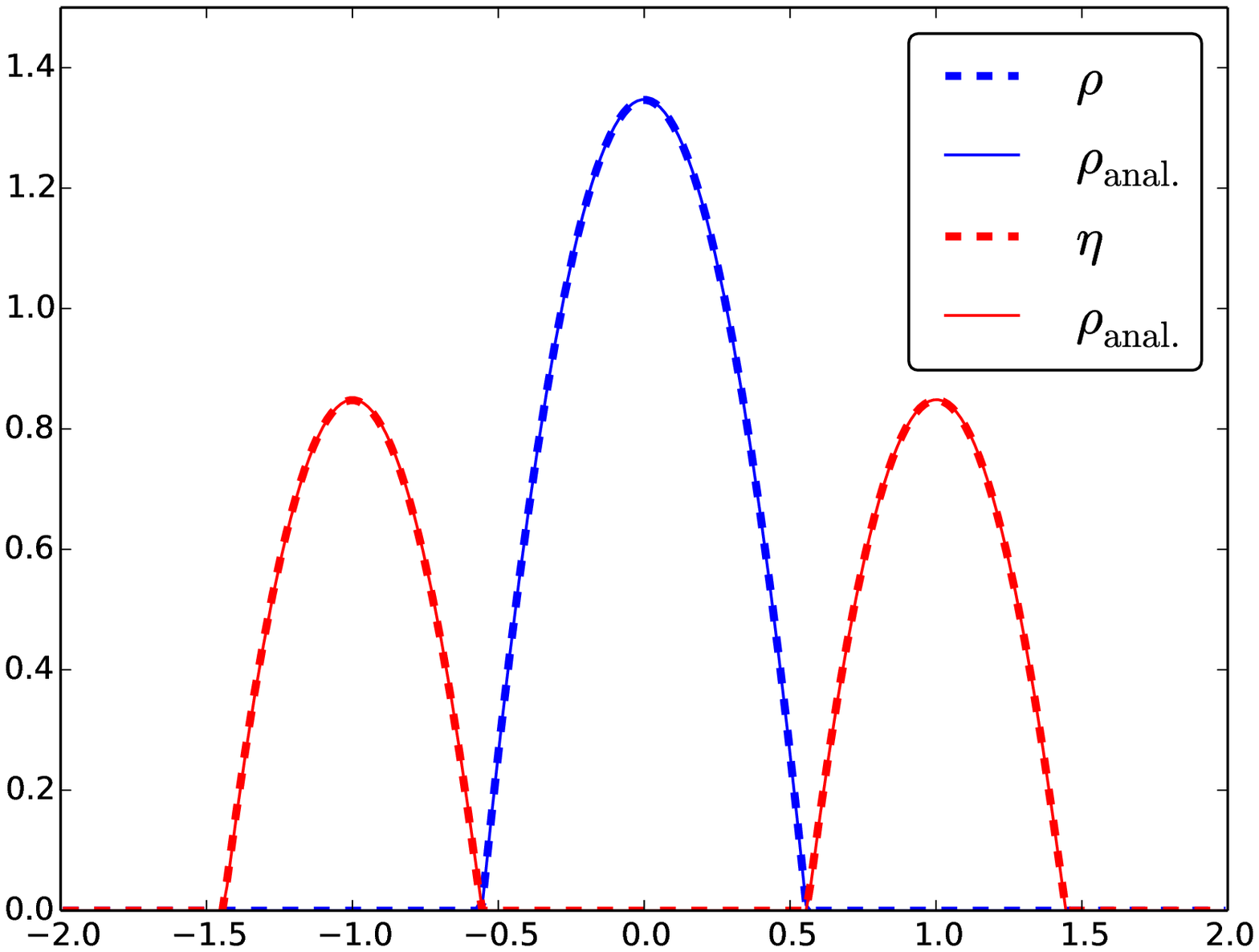}
	}
	\hspace{-1em}
	\subfigure[$\epsilon = 1/2$.]{
		\includegraphics[width=0.31\textwidth]{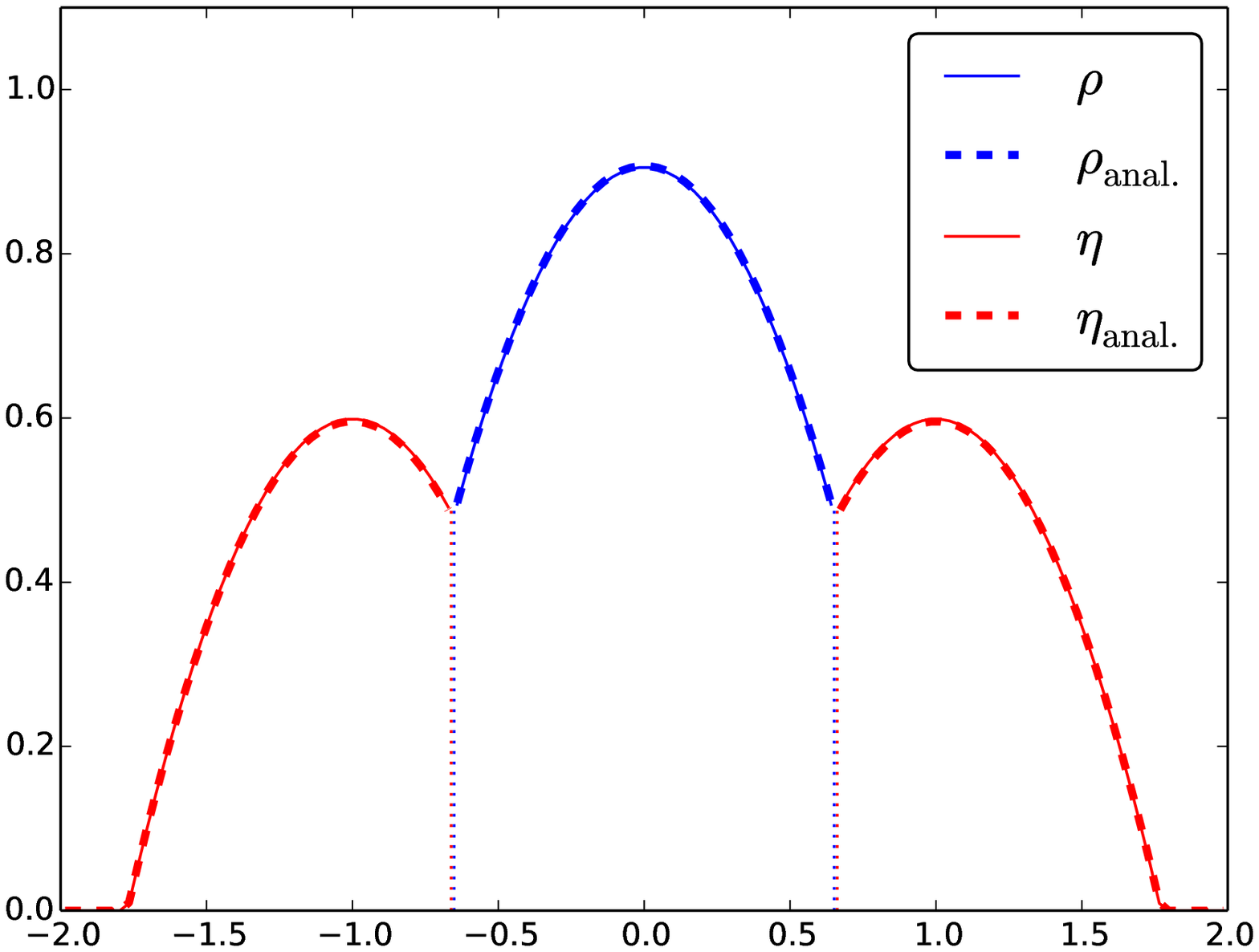}
	}
	\caption{Stationary distributions for same initial data and mass. For $\epsilon < \epsilon_c$ both species are strictly segregated. For $\epsilon> \epsilon_c$ both species are adjacent and they touch at $x=\pm c$ for the critical value of $\epsilon$.}
	\label{fig:AttrRepForDifferentEpsilon}
\end{figure}

Figure \ref{fig:AttrRepForDifferentEpsilon} displays the steady states in the symmetric case, \textit{i.e.} $M_2 = 0$, for attractive-repulsive cross-interactions. We observe a transition of behaviour for different values of $\epsilon$, ranging from strictly segregated states to completely adjacent states. The numerical results agree perfectly with the results obtained analytically.\\

\paragraph{Vanishing diffusion regime}
As we have seen in Figure \ref{fig:CritEpsMaxM2}, there is an $\epsilon_c$ such that the steady states parameterised by $M_2 \in [-M_{2,\rm{max}}, M_{2,\rm{max}}]$ are segregated. In this section we consider the case of vanishing cross-diffusion. We can assume that $\epsilon<\epsilon_c$ and $M_2\in [M_{2,\rm{min}}, M_{2,\rm{max}}]$. Then Eqs. \eqref{eq:support_of_segregated_attrrep_steadystates} determine the support of all densities. We can rewrite the support as follows
\begin{align*}
	\supp(\eta^L) &= \frac{M_2 -m_1}{m_2} - (3/4 \epsilon)^{1/3}\, [-1,1],\\
	\supp(\rho)   &= (3/4 \epsilon)^{1/3}\, [-1,1],\\
	\supp(\eta^R) &= \frac{M_2 +m_1}{m_2} + (3/4 \epsilon)^{1/3}\, [-1,1] 
\end{align*}
We see that the support shrinks to the single points
\begin{align*}
	\supp(\eta^L)=\left\{\frac{M_2-m_1}{m_2}\right\},\quad \supp(\rho) = \{0\}, \qand \supp(\eta^R) = \left\{\frac{M_2+m_1}{m_2}\right\},
\end{align*}
and $M_2\in(-m_1,m_1)$ with steady states 
\begin{align}
	\label{eq:eps2zero_attrrep}
	\rho = \delta_0 \qqand \eta = \frac12 \left(\delta_{\frac{M_2-m_1}{m_2}}+\delta_{\frac{M_2+m_1}{m_2}}\right).
\end{align}
This is indeed a measure solution of system. To see this let us consider
\begin{align*}
	\dot X   &= -\frac{m_2}{2} W_{12}'(X-Y_1) - \frac{m_2}{2}W_{12}'(X-Y_2),\\
	\dot Y_1 &= -\frac{m_2}{2} W_{22}'(Y_1-Y_2) - m_1 W_{21}'(Y_1-X),\\
	\dot Y_2 &= -\frac{m_2}{2} W_{22}'(Y_2-Y_1) - m_1 W_{21}'(Y_2-X).
\end{align*}
Since we are looking for a steady state we observe 
\begin{align*}
	\dot X  = 0 \quad \Leftrightarrow\quad \left\{
	\begin{array}{l}
		X-Y_1 > 0\, \land\, X-Y_2<0,\quad \text{or} \\
		X-Y_1 < 0\, \land\, X-Y_2>0.
	\end{array}
	\right.
\end{align*}
We assume without loss of generality that $X-Y_1 >0\, \land\, X-Y_2<0$, \textit{i.e.} $Y_1 < X < Y_2$. From $\dot Y_1 =0 $ a short computation yields
\begin{align*}
	Y_2 - Y_1 = 2\frac{m_1}{m_2}.
\end{align*}
Fixing $X=0$ we get $Y_2 = Y_1 + 2 m_1/m_2$ and $Y_1 \in [-2\frac{m_1}{m_2},0]$. This is exactly the solution of the system as $\epsilon \rightarrow 0$, cf. Eq. \eqref{eq:eps2zero_attrrep}.\\

\paragraph{Stability of steady states}
Here we want to discuss the stability of the stationary states of the attractive-repulsive system. In general, the stationary states are not stable as small perturbations may lead to a completely different stationary state. It becomes clear in Figure \ref{fig:CritEpsMaxM2}, that perturbing $\eta$ by shifting it to either side leads to a completely different stationary state. Although this is an arbitrarily small perturbation in any $L^p$-norm, the translated profile is another stationary state. This is why these profiles are not stable. The same argument holds for symmetric stationary states. However, they are stable under symmetric perturbations since any symmetric initial data is attracted by the symmetric profile. Characterising fully the basin of attraction for each stationary state seems difficult. For perturbations shifting mass from $\eta^L$ to $\eta^R$ (or vice versa) there is no stationary state but the profile is then attracted by a travelling pulse solution.

\subsubsection{\label{sec:TravellingPulses}Travelling pulses}
In addition to the convergence to steady states we observe travelling pulse solutions in the case of attractive-repulsive cross-interactions. There are two types of travelling pulses  -- those consisting of two bumps and those consisting of three. 

In our numerical study we do not observe more than three bumps, even in the case of exponentially decaying potentials. There are however metastable states where more bumps exist but after a sufficiently long time the collapse into two or three. 

\paragraph{Two pulses}
In order to compute these profiles, we assume $[-a,a]$ denotes the initial support of $u = \eta(0)$ and therefore $[-a-x_0,a-x_0]$ the initial support of $\rho(0)$.

We transform the system into co-moving coordinates, $z= x-vt$, and obtain the following conditions for the pulse profiles 
\begin{align*}
	\begin{split}
	\begin{array}{rll}
	c_1 &= (W_{11}\star u)(z+x_0) + (W_{12}\star u)(z) + \epsilon u(z+x_0) +v z, &\text{ on } [-a-x_0,a-x_0],\\
	c_2 &= (W_{22}\star u)(z) + (W_{21}\star u)(z+x_0) + \epsilon u(z) +v z, &\text{ on } [-a,a].
	\end{array}
	\end{split}	
\end{align*}
similarly to Eqs. \eqref{eq:steadystatecond}. A computation similar to Eqs. (\ref{eq:SelfInteractionTerm}, \ref{eq:CrossInteractionTerm}), leads to the explicit form of the pulse 
\begin{align*}
	u(z) = -\frac{1}{2\epsilon} mz^2 + \frac{M+m-v}{\epsilon} z + \tilde c_1,
\end{align*}
on $[-a,a]$ for some constant $\tilde c_1$. Since $u(z)$ is a parabola with roots $\pm a$, $u$ is  symmetric. As a consequence we obtain $M=v-m$. By definition of $M=\int z u(z)\d z = 0$, whence	$v = m$. Hence the shape is given by
\begin{align}
	\label{eq:TwoPulseProfile}
 	u(z) = u(z)- u(a) = -\frac{1}{2\epsilon}m(z^2-a^2).
 \end{align}
Then the following consideration determines the boundary of the support, $a$,
\begin{align}
	\label{eq:19082016_1624}
	\int_{-a}^a u(z)\d z = m \quad \Rightarrow\quad a = \sqrt[3]{\frac{3\epsilon}{2}}.
\end{align}
Finally, the distance between both profiles, $x_0$, is arbitrary so long as it does not lead to an overlap of both pulses, \textit{i.e.} $x_0 \geq 2a$, because both profiles are moving at the same speed.

Lastly, let us show that there are no adjacent solutions that is solutions whose support is of the form
\begin{align*}
	\supp(u_1) = [-a,0] \qquad \text{and} \qquad \supp(u_2) = [0,a].
\end{align*}
If there were travelling pulse solutions of this form they would satisfy the same equations as above. Then,
\begin{align*}
	\begin{array}{rll}
		u_1(z) &= -\dfrac{1}{2\epsilon} m(z^2 -a^2) + \dfrac{M+m-v}{\epsilon} (z+a), & \text{ on } [-a,0],\\[1em]
		u_2(z) &= -\dfrac{1}{2\epsilon} m(z^2 -a^2) + \dfrac{M+m-v}{\epsilon} (z-a), & \text{ on } [0,a].
	\end{array}
\end{align*}
The continuity of the sum suggests that $u_1(0)=u_2(0)$ implies $m=v$. But then 
\begin{align*}
	\int_0^a z u_2(z) \d z = M \qquad \Rightarrow \qquad M = \frac34{\frac {m{a}^{4}}{5\,{a}^{3}-6\,\epsilon}}
\end{align*}
We solve this expression for $a>0$ and find
$a = \sqrt[3]{12\epsilon}$.
A comparison of the support of the adjacent solutions and the support of segregated solutions, cf. Eq. \eqref{eq:19082016_1624}, shows that the adjacent solutions in fact only touch.

\begin{figure}[!ht] 
	\includegraphics[width=1\textwidth]{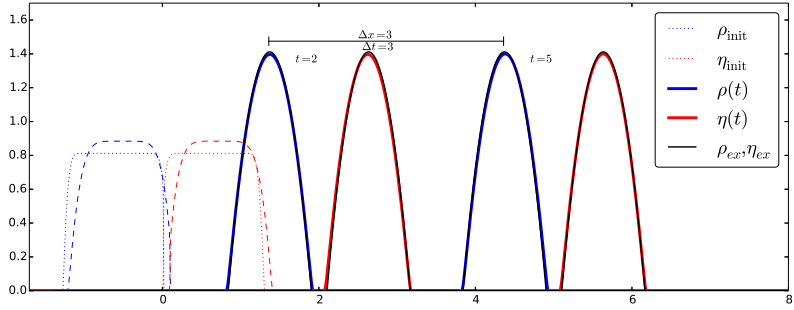}	
	\caption{The initial data is given by the dotted graph, the dashed lines are an intermediate solution. After some time $t\approx 2$ the travelling pulse profile is established. The pulses move at velocity $v=m=1$ as can be seen in the graph, since $\Delta x = \Delta t = 3$.}
	\label{fig:twopulsesolution}
\end{figure}

Figure \ref{fig:twopulsesolution} shows the formation of two travelling pulses. We start with two indicator functions as initial data and let the system evolve. At about time $t\approx 2$ we observe a fully established pulse profile. We let the system evolve further and compare the solution at $t=5$ with the solution at time $t=2$. The figure shows that the shapes do not change any further but are only transported at a velocity of $v=\Delta x / \Delta t =1$ in perfect agreement with the analytical result, $v=m$.

Subsequently, we shall see that the solution consisting of two pulses is in fact a special case of the three-pulses configurations. The latter consist of the first species, $\rho$, surrounded by the second species, $\eta$. We assume
\begin{align*}
	\supp(\rho) = [-c,c], \qquad \text{and} \qquad \supp(\eta) = [a,b] \cup [d,e],
\end{align*}
where $a < b \leq -c <c \leq d <e$ are real numbers and
\begin{align*}
	\rho(t,x) = \rho(x-vt),\qquad \text{and} \qquad \eta(t,x) = \eta(x-vt).
\end{align*}
We transform to co-moving coordinates, $z = x-vt$, and obtain the following conditions for the profile
\begin{align*}
	\begin{array}{rll}
	c_1 &= (W_{11} \star \rho)(z) + (W_{12}\star\eta)(z) + \epsilon \rho(z) + vz, & \text{on } [-c,c],\\
	c_2 &= (W_{22} \star \eta)(z) + (W_{21}\star\rho)(z) + \epsilon \eta(z) + vz, & \text{on } [a,b]\cup[d,e],\\
	\end{array}
\end{align*}
whence we obtain
\begin{align}
	\label{eq:ThreePulseProfileRho}
	\rho(z) = - \frac{1}{2\epsilon} m z^2 + \frac{(m^R-m^L) -v}{\epsilon}z + \tilde c_1.
\end{align}
Here
\begin{align}
	\label{eq:massofpulses}
	m^L = \int_a^b \eta(z) \d z,\qqand m^R = \int_d^e \eta(z) \d z.
\end{align}
Similarly, the profiles of the second species are given by
\begin{align*}
	\eta^L(z) = -\frac{1}{2\epsilon}m (z^2 -a^2) + \frac{M_2 -m -v}{\epsilon}(z-a), \qquad \text{on } [a,b],
\end{align*}
and
\begin{align*}
	\eta^R(z) = -\frac{1}{2\epsilon}m (z^2 -e^2) + \frac{M_2 +m -v}{\epsilon}(z-e), \qquad \text{on } [d,e].
\end{align*}
Again, we use the fact that the sum of both densities has to be continuous, \textit{i.e.}
\begin{align*}
	\eta^L(a) = 0, \quad \eta^L(b) = \rho(-c), \quad \rho(c) = \eta^R(d), \quad \text{and} \quad \eta^R(e) = 0,
\end{align*}
where $\rho(-c) = 0$ if $b<-c$, and $\rho(c) = 0$ if $c<d$. In addition the conditions on the masses
\begin{align}
	\label{eq:masscondition}
	m = \int_{-c}^c \rho(z)\d z,
\end{align}
as well as Eqs. \eqref{eq:massofpulses} hold. We consider the case of strictly segregated solutions first, \textit{i.e.} $b<-c$, and $c<d$. 
Since then $\rho(\pm c) = 0$, we may deduce from Eq. \eqref{eq:ThreePulseProfileRho} that  $v = m^R-m^L$ for the speed of propagation and
\begin{subequations}
\label{eq:support_and_profiles_of_travellingpulses}
\begin{align*}
	\rho(z)= -\frac{1}{2\epsilon} m (z^2 -c^2), \qquad \text{with}\; c = \sqrt[3]{\frac32 \epsilon},
\end{align*} 
for the shape of the first species ($c$ is determined by the mass condition, Eq. \eqref{eq:masscondition}). Furthermore we obtain
\begin{align}
	\label{eq:15092016_1539}
	\eta^L(b)= 0 \quad \Rightarrow \quad a=\frac {-m b+2{M_2}-2m-2v}{{m}},
\end{align}
in terms of $b$. Similarly, we can get an expression for $e$ in terms of $d$, \textit{i.e.}
\begin{align}
	\label{eq:15092016_1544}
	\eta^R(d) = 0\quad \Rightarrow \quad e = \frac {-{m} d+2 {M_2}+2{m}-2v}{{m}}.
\end{align}
Using the expression for $a$, we obtain
\begin{align}
	\label{eq:15092016_1535}
	\int_a^b\eta^L(x) \d x = m^L \quad \Rightarrow \quad b = {\frac {\sqrt [3]{\dfrac32 \epsilon m^2
m^L}+{M_2}-{m}-v}{{m}}}.
\end{align}
Now we employ the expression for $e$ to get
\begin{align}
	\label{eq:15092016_1535_1}
	\int_d^e \eta^R(x) \d x= m^R	\quad \Rightarrow \quad d= {\frac {-\sqrt [3]{\dfrac32 \epsilon  m^2 m^R}+{M_2}+{m}-v}{{m}}}.
\end{align}
\end{subequations}
Note that Eqs. \eqref{eq:support_and_profiles_of_travellingpulses} completely determine the support and the profiles of the pulses. Figure \ref{fig:threepulsesolution} shows the formation of a triple pulse solution. We choose characteristic functions as initial data (dotted). The mass on the left is $m^L = 1/3$ and, respectively, $m^R=2/3$ on the right. After some time the pulse profile is established. We compare the system (blue and red) at time $t=9$ and time $t=24$ with the analytical expression derived above (black). The figure displays a great agreement between our numerical result and the analytical. Once the profile is fully established it moves to the right at a constant speed. The numerical velocity is given by $\Delta x/\Delta t = 5/15 = 1/3$. This is in perfect agreement with the analytically obtained results, \textit{i.e.} $v=m^R-m^L = 2/3 -1/3 = 1/3$.\\
\begin{figure}[ht!]
	\centering
	\includegraphics[width=0.85\textwidth]{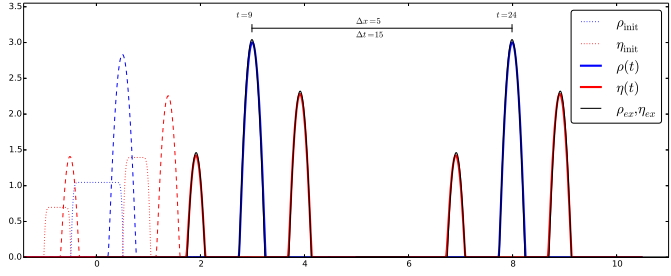}
	\caption{Travelling pulse solution consisting of three pulses. The mass of the second species amounts to $m^L = 1/3$ on the left and to $m^R=2/3$ on the right, respectively. The dotted lines represent the initial data and the dashed lines represent an intermediate solutions. After some time, $t\approx 9$, the pulse profiles have established and both species move to the right at a velocity of $v=\Delta x/\Delta t = 1/3$. The thick red and blue lines correspond to the numerical solution, the black lines to the analytically computed travelling pulse solution.}
	\label{fig:threepulsesolution}
\end{figure}

At this stage, let us draw our attention to two special cases. 
\begin{remark}
First, we consider the case $m^L = 0$. In this case $v=m^R-m^L = m$. Thus, in conjunction with Eqs. (\ref{eq:15092016_1539},\ref{eq:15092016_1535}) there holds $b = \tfrac{M_2}{m} = a$,
\textit{i.e.} the left part is degenerate.  Moreover, the support of the right part is
\begin{align*}
 d= -\sqrt [3]{\dfrac32 \epsilon} + \frac {M_2}{m},
\end{align*}
according to Eq. \eqref{eq:15092016_1535_1}. We substitute this into the Eq. \eqref{eq:15092016_1544} and get
\begin{align*}
	e = \frac {-m d+2 M_2}{m} =  \sqrt [3]{\dfrac32 \epsilon} + \frac{M_2}{m}.
\end{align*}
Thus we can write the support in the following form
$
	\left[-\sqrt[3]{\frac32 \epsilon},\sqrt[3]{\frac32 \epsilon}\right] + \frac{M_2}{m}.
$
Let us have a closer look at $\eta^R$ now. Using $e = c + \frac{M_2}{m}$ we obtain
\begin{align*}
	\eta^R(z) &= -\frac{1}{2\epsilon}m \bigg(z - \frac{M_2}{m} -c \bigg)\!\bigg( z- \frac{M_2}{m} + c\bigg),
\end{align*}
where we set $x_0 := M_2/m$. Thus we finally obtain 
\begin{align*}
	\eta^R(z) = -\frac{1}{2\epsilon} m ((z-x_0)^2-c^2),
\end{align*}
supported on the interval $[-c+x_0, c+x_0]$. This is precisely the solution to the two-pulse system, cf. Eq. \eqref{eq:TwoPulseProfile}. 
\end{remark}

\begin{remark}
The second remark concerns the case $m^R = m^L$. Then $v=0$ and, according to Eqs. (\ref{eq:15092016_1535}, \ref{eq:15092016_1535_1}), we get
\begin{align*}
	b = \left(\sqrt [3]{\dfrac34 \epsilon} - 1\right) + \frac{M_2}{m}\quad \text{and}\quad d= - \left(\sqrt [3]{\dfrac34 \epsilon}-1\right) +\frac{M_2}{m},
\end{align*}
which are equal to Eqs. (\ref{eq:30082016_0940}, \ref{eq:15092016_2014}) in the case $m=m_1=m_2$. In addition, Eqs. (\ref{eq:15092016_1539},\ref{eq:15092016_1544}) turn into Eqs. (\ref{eq:2408216_1229}, \ref{eq:240816_1425}), \textit{i.e.} the support of the tripple pulse solutions agrees with the support of the fully segregated steady states. Similarly, the shapes agree in the case $v=0$.
\end{remark}

\begin{remark}[Maximal $M_2$]
Let us get back to the general case. We study the interval of $M_2$. Assuming $\epsilon$ fixed, $b=-c$ yields
\begin{align*}
	M_{2,\rm{max}} &= \frac12 \sqrt [3]{12\epsilon}m +\frac12 \sqrt [3]{12{m_R}\,\epsilon\,m^2}-m_1+v.
\end{align*}
On the other hand, $c=d$ gives
\begin{align*}
	M_{2,\rm{min}}=-\frac12\sqrt [3]{12\epsilon}{m}-\frac12 \sqrt [3]{12{m_L}\,\epsilon\,m^2}+m+v,
\end{align*}
where $v=m^R-m^L$, as above. It is worthwhile noting that in the case $m^L=m^R$ both $M_{2,\rm{max}}$ and $M_{2,\rm{min}}$ coincide with Eqs. (\ref{eq:16092016_1845}, \ref{eq:16092016_1845_1}) for the stationary state.
\end{remark}

Parallel to the consideration for (partially) adjacent steady states of the attractive-repulsive system we also find the existence of adjacent travelling pulse solutions. 

\section{\label{sec:gen}Generality}
This section is dedicated to the study of more general or realistic potentials to understand whether the behaviours observed above are specific to our interaction potentials. Different cross-interaction and  self-interaction potentials will be investigated.
Even though analytic expressions for the steady states and travelling pulses seem no longer avaiablable, the behaviours are indeed generic and, in fact, even richer than the above particular model. 
\subsection{Different cross-interactions}
Let us begin by considering different cross-interaction potentials. We regard two types of potentials  --- power-laws  and Morse-like potentials decaying at infinity, \textit{i.e.}
\begin{align*}
	W_{12} = |x|^p = \pm W_{21}, \quad \text{and} \quad W_{12} = 1 - \exp(-|x|^p) = \pm W_{21},
\end{align*}
where $p \in\{1/2,1,3/2\}$. This choice of potentials is motivated as they are similar to the Newtonian cross-interaction. 

In both cases, we observe a very similar behaviour both in the mutually attractive case and the attractive-repulsive case, respectively. 
	\begin{figure}[ht!]
		\centering
		\subfigure[$W_{\mathrm{cr}} = |x|^{\frac12}$]{
			\includegraphics[width=0.28\textwidth]{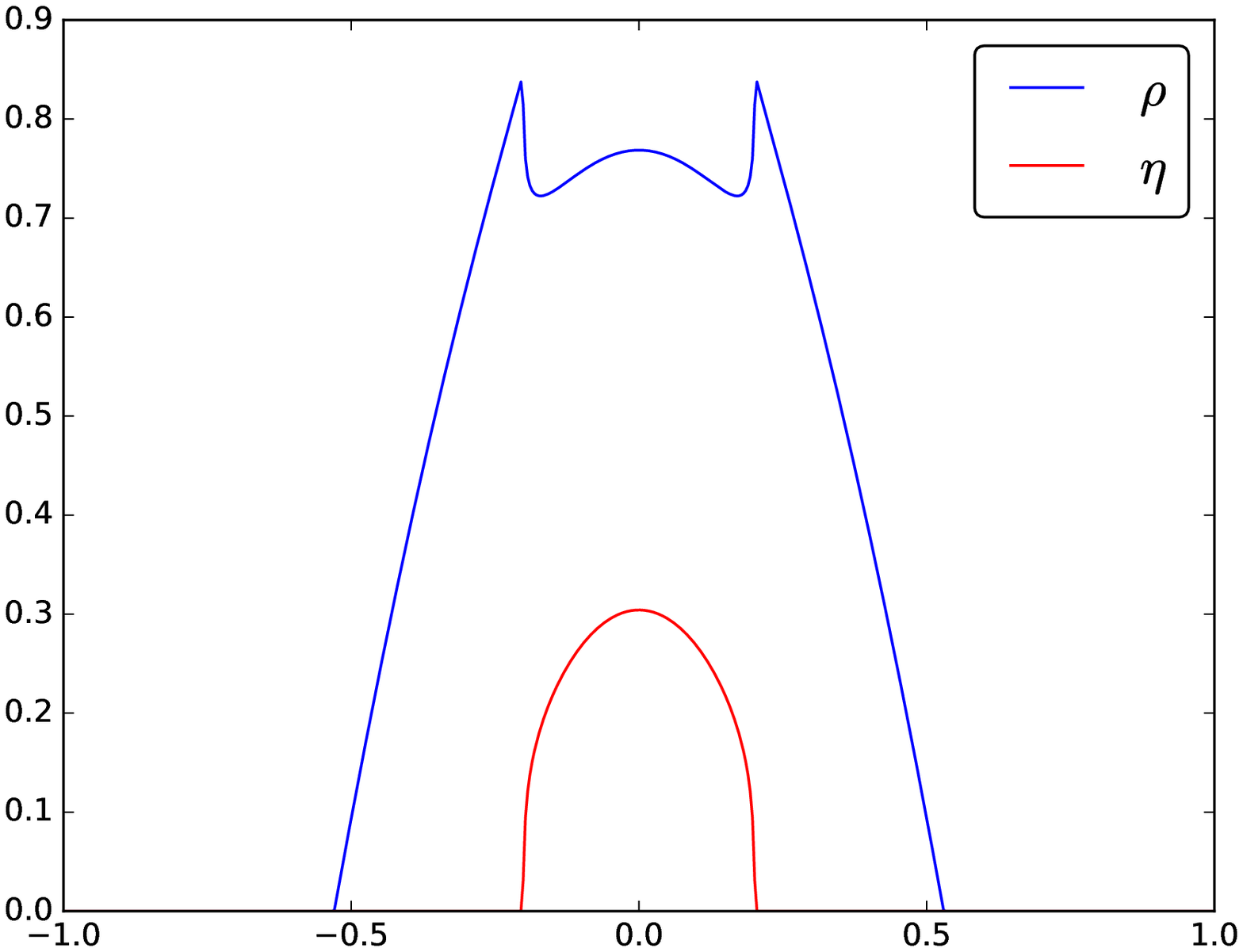}
		}
		\subfigure[$W_{\mathrm{cr}} = |x|$]{
			\includegraphics[width=0.28\textwidth]{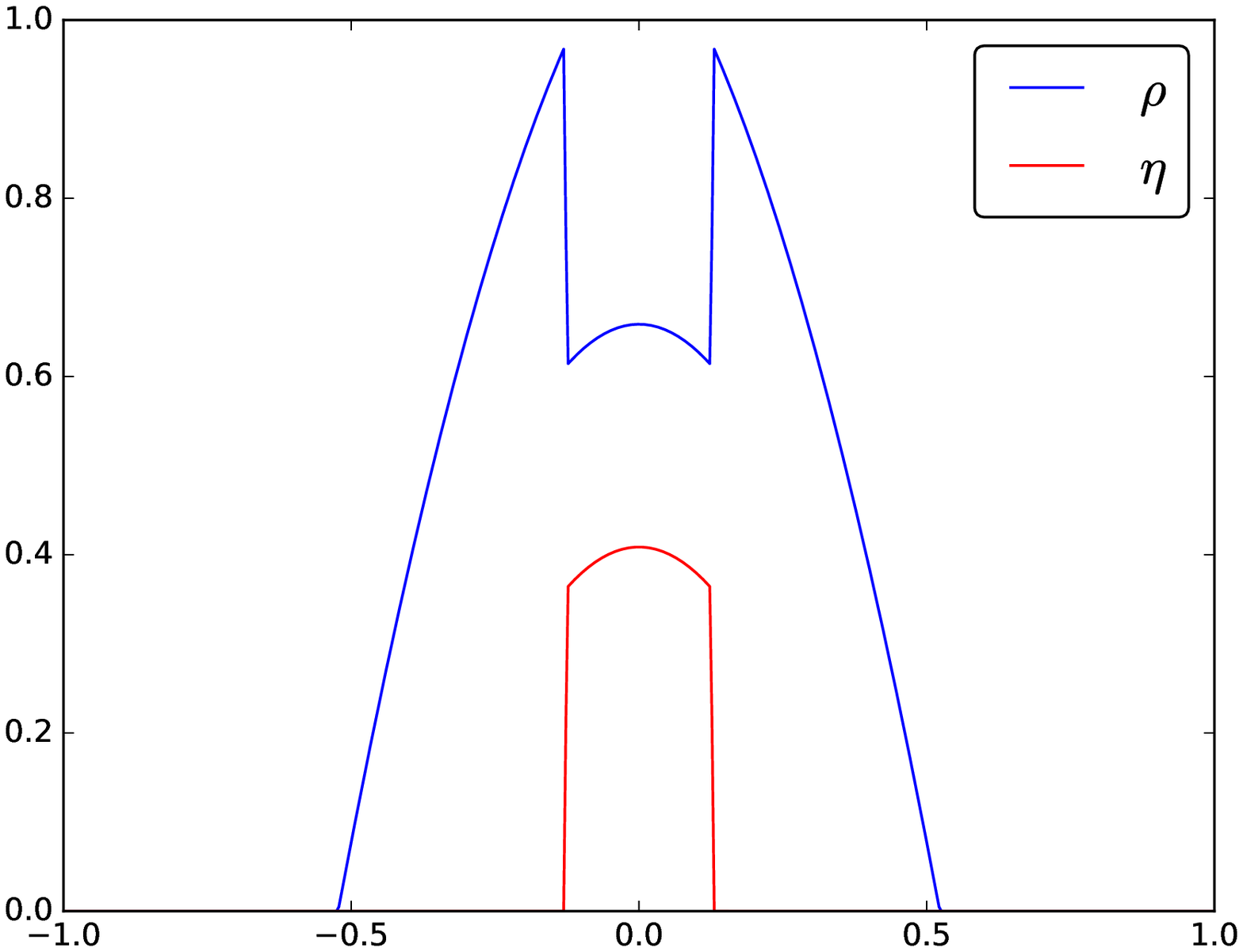}
		}
		\subfigure[$W_{\mathrm{cr}} = |x|^{\frac32}$]{
			\includegraphics[width=0.28\textwidth]{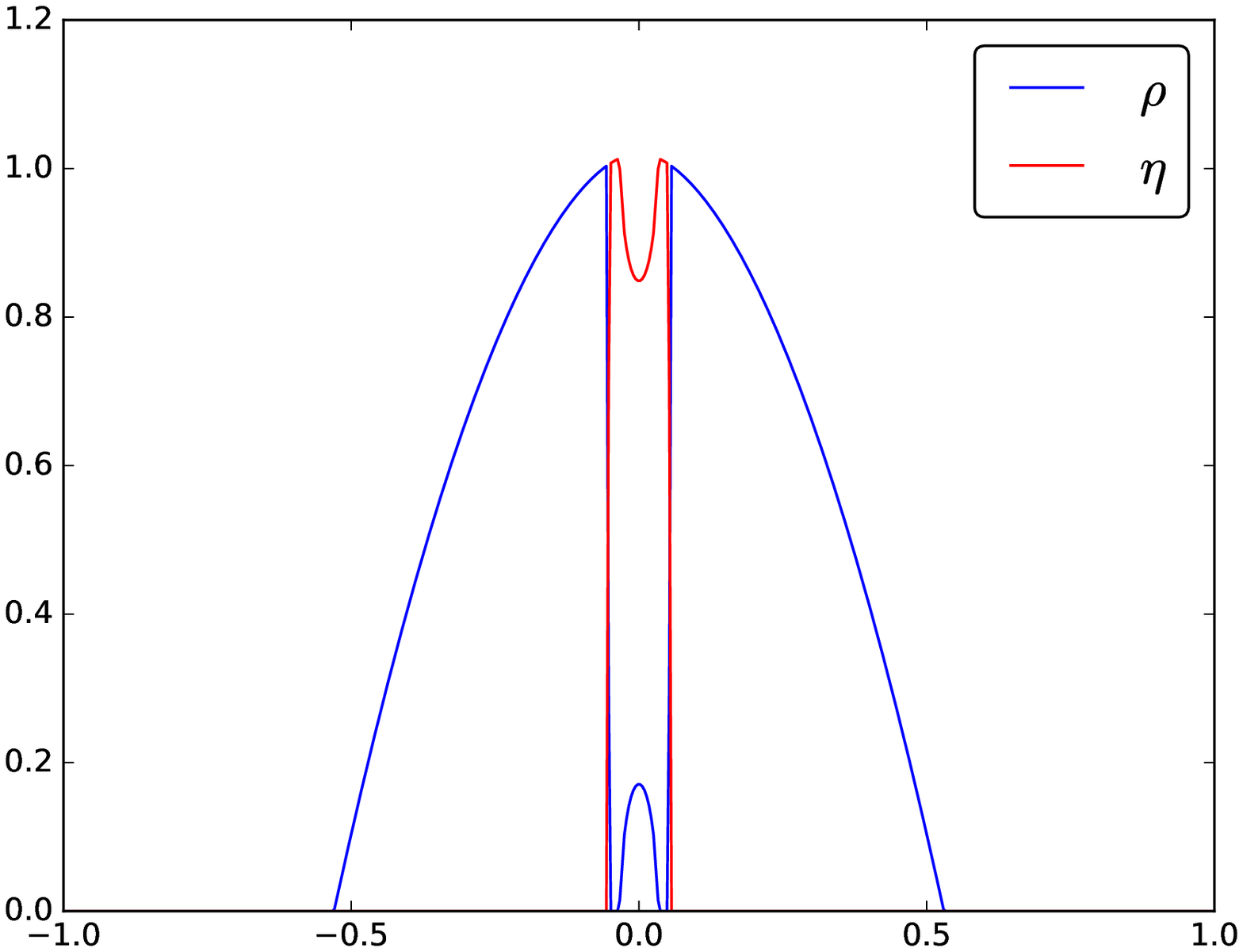}
		}
		\subfigure[$W_{\mathrm{cr}} = 1-\exp(-|x|^{\frac12})$]{
			\includegraphics[width=0.28\textwidth]{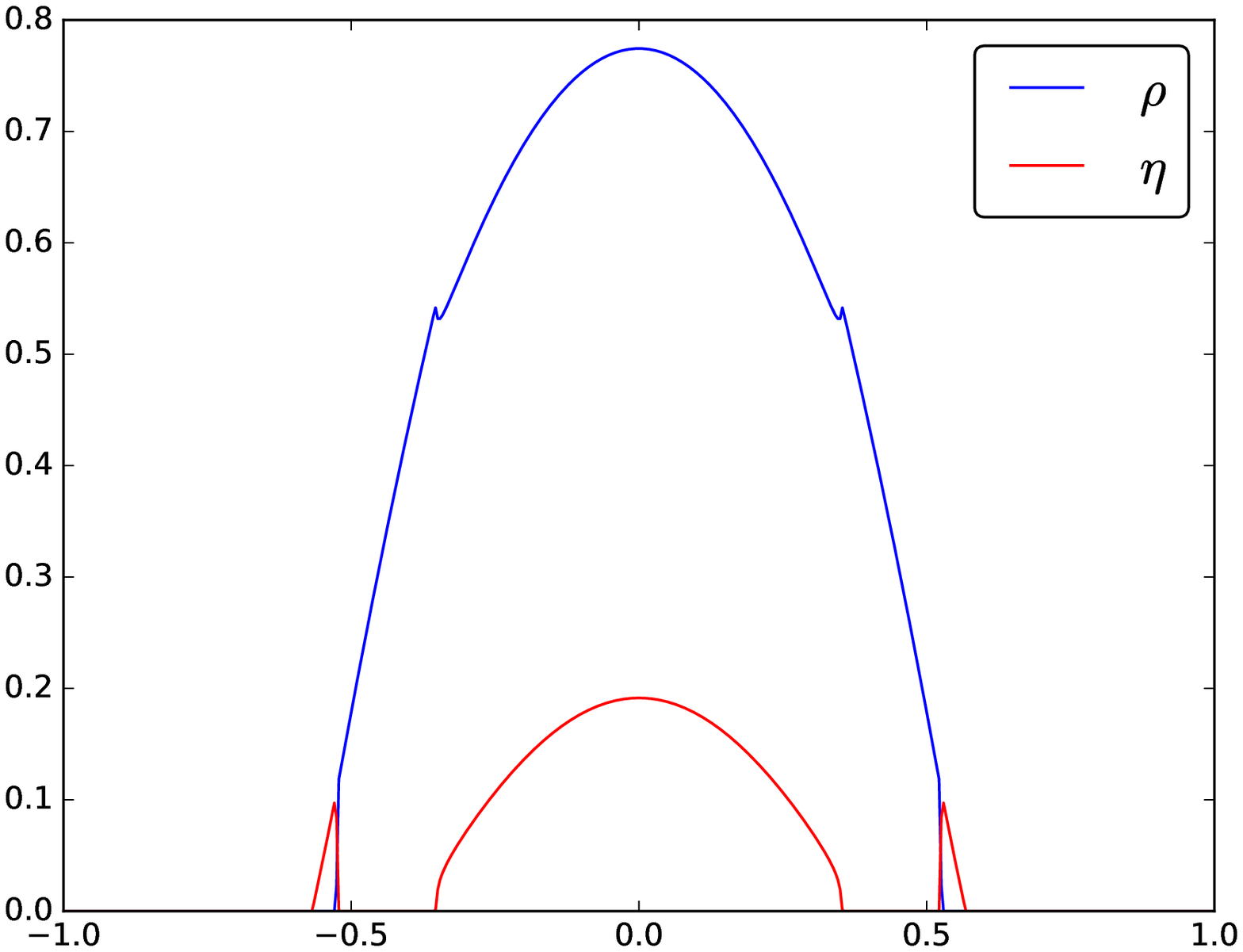}
			\label{subfig:d}
		}
		\subfigure[$W_{\mathrm{cr}} = 1 -\exp(-|x|)$]{
			\includegraphics[width=0.28\textwidth]{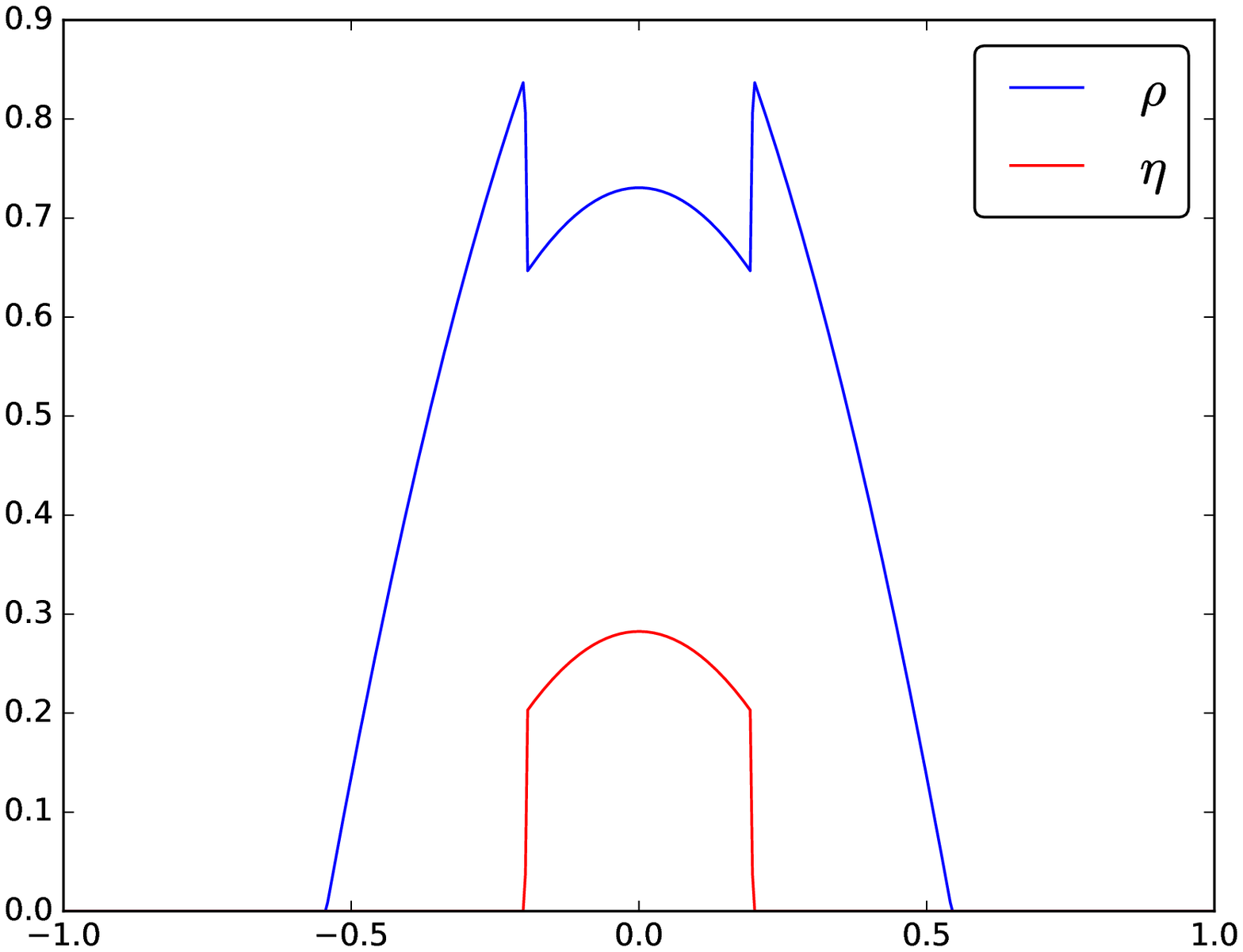}
		}
		\subfigure[$W_{\mathrm{cr}} = 1- \exp(-|x|^{\frac32})$]{
			\includegraphics[width=0.28\textwidth]{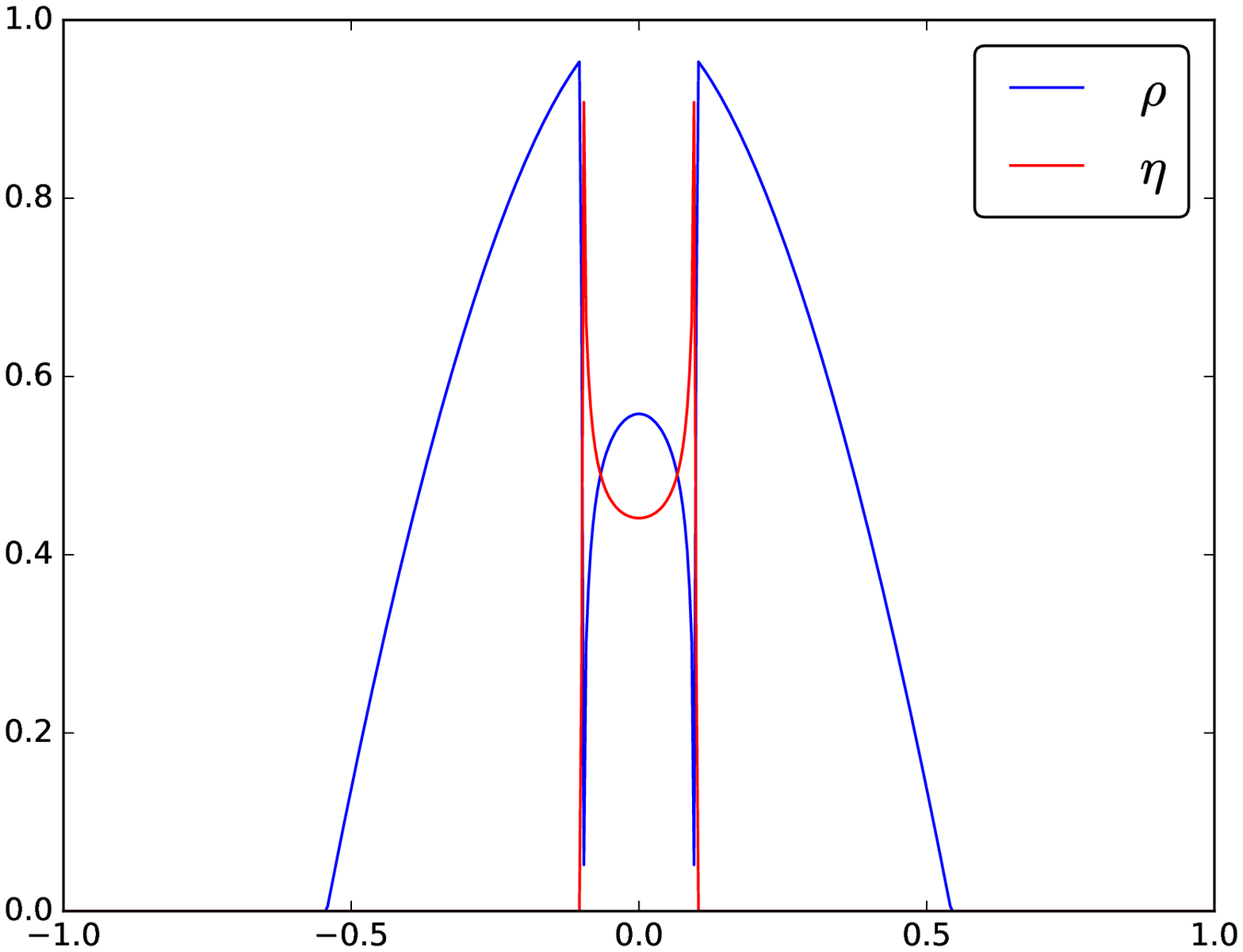}
		}
		\caption{The Batman profiles for different interaction potentials.}
		\label{fig:differentbatmans}
	\end{figure}
	Figure \ref{fig:differentbatmans} displays the Batman profile for different cross-interaction potentials. In all simulations the same initial data, mass, and cross-diffusivity were used. Each steady state features the salient characteristics observed in the case  $W_{cr}=|x|$, \textit{i.e.} a region of coexistence surrounded by regions inhabited by only one species. From the steady states we can also infer another information, namely,  second type profiles exist and the point of bifurcation  depends on the potential, for only Figure   \ref{subfig:d} exhibits a profile of second type. Similarly, we observe a  symmetrising effect for small cross-diffusivities and asymmetric profiles.

\subsection{Different self-interactions}
Here, we keep the cross-interaction potentials fixed as $W_{12} = |x| = \pm W_{21}$ and consider different self-interaction potentials of the form $W(x)=|x|^p/p$, for $p \in \{3/2, 2, 4\}$. In each case we observe a very similar behaviour. We obtain the same variety including both Batman profiles and the profiles of second type. Again we observe that the system is symmetrising, however for a different $\epsilon^{(1)}$. In the attractive-repulsive case as well we observe the characteristic profiles and the formation of pulses.


\section{\label{sec:con}Conclusions}
In this paper we introduced a system of two interacting species with cross-diffusion. We used a positivity-preserving finite-volume scheme in order to study the system numerically. For a specific choice of potentials, the steady states can be constructed with parameters governed by algebraic equations. These numerically simulated and the analytically constructed stationary states and travelling pulses
were found to agree with each other.
Using the same scheme the model was explored for related potentials 
and the behaviours observed for the specific potentials
turned out to be generic, when the cross-interaction potentials
or the self-interaction potentials were exchanged.
While this paper gives a first insight as to what qualitative properties can be expected from models taking the general form~\eqref{eq:OurModelMultiD}, there is still a lot of analytical work to be done. First and foremost, it is still an open problem to show existence of solutions to the systems. The formal gradient flow structure
is lost when the cross-interaction potentials $W_{12}$ and $W_{21}$
are not proportional to each other, and the main problem is to find the right estimates for individual species since we only control
the gradient of the sum of the densities. 


\section*{Acknowledgments}
JAC was partially supported by the Royal Society via a Wolfson Research Merit Award and by EPSRC grant number EP/P031587/1.

\bibliographystyle{abbrv}

\def\cprime{$'$}

\end{document}